\documentclass[11pt,reqno]{amsart}
\usepackage{amsfonts,amsmath}
\usepackage{amssymb}
\usepackage{textcomp}
\usepackage{enumitem}
\usepackage{graphicx,color}
\usepackage{dsfont}
\numberwithin{equation}{section}
\newcommand{\nc}{\newcommand}
\nc{\parent}[1]{$[\![#1]\!]$}
\newtheorem{theorem}{Theorem}[section]
\newtheorem{lemma}{Lemma}[section]

\newtheorem{proposition}{Proposition}[section]
\newtheorem{remark}{Remark}[section]

\newtheorem{assumption}{Assumption}[section]

\newenvironment{pf-main}{{\sc Proof of Theorem \ref{mainresult}.}\hspace{3mm}}{\qed}

\nc{\cadlag}{c\`{a}dl\`{a}g } \nc{\ba}{\begin{array}}
\nc{\ea}{\end{array}} \nc{\be}{\begin{equation}}
\nc{\ee}{\end{equation}} \nc{\bea}{\begin{eqnarray}}
\nc{\eea}{\end{eqnarray}} \nc{\bean}{\begin{eqnarray*}}
\nc{\eean}{\end{eqnarray*}} \nc{\bu}{\bullet} \nc{\nn}{\nonumber}
\nc{\cA}{{\mathcal A}} \nc{\cB}{{\mathcal B}} \nc{\cC}{{\mathcal
C}} \nc{\cD}{{\mathcal D}} \nc{\bbD}{\mathbb{D}}
\nc{\cG}{{\mathcal G}} \nc{\cF}{{\mathcal F}} \nc{\cS}{{\mathcal
S}} \nc{\cU}{{\mathcal U}} \nc{\cH}{{\mathcal H}}
\nc{\cK}{{\mathcal K}}\nc{\cL}{{\mathcal L}}  \nc{\cM}{{\mathcal
M}} \nc{\cO}{{\mathcal O}} \nc{\cT}{{\mathcal T}}  \nc{\cP}{{\mathcal P}} \nc{\cW}{{\mathcal W}}
\nc{\bbE}{\mathbb{E}} \nc{\tbA}{\tilde{\bbA}}\nc{\bbA}{\mathbb{A}}\nc{\bbF}{\mathbb{F}}
\nc{\bbEQ}{\mathbb{E}_{\mathbb{Q}}} \nc{\eps}{\varepsilon}
\nc{\bbEP}{\mathbb{E}_{\mathbb{P}}}\nc{\bbL}{\mathbb{L}}
\nc{\what}{\widehat} \nc{\bbP}{\mathbb{P}} \nc{\bbQ}{\mathbb{Q}}
\nc{\del}{\partial} \nc{\Om}{\Omega} \nc{\om}{\omega}
\nc{\bbR}{\mathbb{R}} \nc{\bbN}{\mathbb{N}} \nc{\fps}{$(\Om, \cF,
(\cF_t)_{t\geq 0}, \bbP)$} \nc{\bbC}{\mathbb{C}}
\nc{\bfr}{\begin{flushright}} \nc{\efr}{\end{flushright}}
\nc{\dXt}{\Delta X_{t}} \nc{\dXs}{\Delta X_{s}}
\nc{\bs}{\blacksquare} \nc{\dX}{\Delta X} \nc{\dY}{\Delta Y}
\nc{\dnkx}{\left(X(T^{n}_{k})-X(T^{n}_{k-1})\right)}
\nc{\esssup}{\mathrm{ess}\mbox{ }\mathrm{sup}}
\nc{\essinf}{\mathrm{ess}\mbox{ } \mathrm{inf}}
\nc{\dhats}{\widehat{\delta_s}} \nc{\half} {\frac{1}{2}}
\nc{\elr}{(\ell,r)}
\nc{\wXn}{\what{X}_{t_n}}\nc{\wXnp}{\what{X}_{t_{n+1}}}\nc{\wXt}{\what{X}_{t}}\nc{\wXs}{\what{X}_{s}}
\nc{\wYn}{\what{Y}_{t_n}}\nc{\wYnp}{\what{Y}_{t_{n+1}}}\nc{\wYt}{\what{Y}_{t}}

\nc{\ol}{\overline}
\def\rar{\rightarrow}

\nc{\chf}{\mbox{$\mathbf1$}}
\begin{document}

\title{Speeding up the  Euler scheme for killed diffusions}
\author{Umut \c{C}et\.in}
\address{Department of Statistics, London School of Economics and Political Science, 10 Houghton st, London, WC2A 2AE, UK}
\email{u.cetin@lse.ac.uk}
\author{Julien Hok}
\address{Investec Bank, 30 Gresham St, London EC2V 7QN}
\email{julienhok@yahoo.fr}
\address{}
\date{\today}
\begin{abstract}
Let $X$ be a linear diffusion taking values in  $(\ell,r)$ and consider the standard Euler scheme to compute an approximation to $\bbE[g(X_T)\chf_{[T<\zeta]}]$ for a given function $g$ and a deterministic $T$, where $\zeta=\inf\{t\geq 0: X_t \notin (\ell,r)\}$. It is well-known since \cite{GobetKilled} that the presence of killing introduces a loss of accuracy and reduces the weak convergence rate to $1/\sqrt{N}$ with $N$ being the number of discretisatons. We introduce a drift-implicit Euler method to bring the convergence rate back to $1/N$, i.e. the optimal rate in the absence of killing, using the theory of recurrent transformations developed in \cite{rectr}. Although the current setup assumes a one-dimensional setting, multidimensional extension is within reach as soon as a systematic treatment of recurrent transformations is available in higher dimensions.

\end{abstract}
\maketitle

\section{Introduction} 
Let $X$ be a diffusion on some filtered probability space taking values in  $(\ell,r)$  and solving
\be \label{i:X}
X_t=x+ \int_0^t\sigma(X_s)dB_s+ \int_0^tb(X_s)ds, \quad t <\zeta,
\ee
where $B$ is a Brownian motion, and $\zeta:=\inf\{t\geq 0: X_t \notin (\ell,r)\}$ is the first exit time from the interval $(\ell,r)$. The process is {\em killed} at $\zeta$ and sent to a cemetery state. 

Let's assume that at least one of the boundaries are accessible and $\zeta$ is finite a.s. and consider $\bbE[g(X_T)\chf_{[T<\zeta]}]$ for a given function $g$ and a deterministic $T$. Letting $g$ take the value $0$ at the cemetery state, one can rewrite this expression in terms of the {\em killed} diffusion as  $\bbE[g(X_T)\chf_{[T<\zeta]}]$.  Such computations appear very naturally in many applied problems of science, engineering, and finance.  For instance, in Mathematical Finance theory, such an expectation corresponds to  the price of a barrier option with payoff $g$ and maturity $T$ written on a stock whose price process is given by $X$. The barrier feature renders the option worthless if the stock price hits one of the accessible boundaries before the maturity of the option.

A closed form expression for $\bbE[g(X_T)\chf_{[T<\zeta]}]$ is rarely available even in this one-dimensional setting. Thus, one needs to resort to an approximation scheme for an answer. Arguably the easiest approach is to run a standard Euler-Maruyama scheme on the SDE (\ref{i:X}) by setting
\[
\bar{X}_{t_{n+1}} =\bar{X}_{t_n} +\sigma(\bar{X}_{t_n})(B_{t_{n+1}}-B_{t_n}) + b(\bar{X}_{t_n}) \frac{T}{N},\]
where $\bar{X_0}=x$, $t_0=0$, $N>0$ is an integer, $t_n=\frac{nT}{N}$ for $n=1, \ldots N$, and compute $\bbE[g(\bar{X}_T)\chf_{[T<\tau]}]$, where $\tau$ is the first time that the discrete-time process $(\bar{X}_{t_{n}})_{n=0}^N$ hits any of the barriers. Under standard regularity conditions on the diffusion process and $g$, such a scheme indeed converges as $N\rar \infty$. However, it converges at a rate much slower than a standard Euler-Maruyama scheme applied to a diffusion process that is not killed at accessible boundaries.

Indeed it was shown by Gobet \cite{GobetKilled} that under standard hypothesis the above  scheme for the killed diffusion converges weakly at rate $N^{-1/2}$ as opposed to $N^{-1}$, which is the rate of weak convergence for the  Euler-Maruyama scheme in the absence of killing (see, e.g., Talay and Tubaro \cite{TTeuler} or Mikulevi\v{c}ius and Platen \cite{MPEuler}). This rate is optimal since it is reached when $X$ is a Brownian motion and $g$ is an indicator function of a set strictly contained in $(\ell,r)$ (see Siegmund and Yuh \cite{SYbmkilled}).

\c{C}etin \cite{rectr} conjectured that using a {\em recurrent transformation} would bring the convergence rate back to $N^{-1}$. A recurrent transformation at heart is a change of measure that keeps the Markovian structure intact while transforming the process into a recurrent one. In particular $X$ never touches the boundaries of $(\ell,r)$ under the new measure $\bbQ$.  \cite{rectr} shows that $\bbQ$ is locally absolutely continuous with respect to the original measure $\bbP$, and $X$ follows
\be \label{i:Xrtr}
dX_t=\sigma(X_t)dW_t + \left(b(X_t)+\sigma^2(X_t)\frac{h'}{h}(X_t)\right)dt,
\ee
for some function $h$ and a $\bbQ$-Brownian motion $W$. That the above claim was a conjecture and not following immediately from the standard results on Euler-Maruyama schemes is that $\frac{h'}{h}$ is explosive near boundaries and is not Lipschitz, which is in fact needed for $X$ not to  touch the previously accessible boundaries after the measure change. This can create significant difficulties with approximation and may even lead to divergence (see, e.g., the potential issues that may arise with non-Lipschitz drivers and methods on how to resolve them  in \cite{HJK1} and \cite{HJK2}).

In this paper we prove this conjecture with a slight ``twist.'' Note that if one applies the Euler-Maruyama scheme naively to (\ref{i:Xrtr}), one obtains, as usual, a Brownian motion with drift whose parameters change at times of discretisation. This process will hit finite boundaries with positive probability, and therefore will exit the state space of $X$ with positive probability. One way to overcome this is to impose an ad hoc reflection on the boundaries. However, this will introduce a local time term in computations requiring additional estimates on its convergence rate to $0$. Moreover, it is far from obvious that reflection is the optimal resolution of problems arising from the discretised process exiting the domain. 

We instead study a drift-implicit method that keeps the state space intact after discretisation. To see this, suppose  that $b\equiv 0$, which can be obtained by changing the scale if necessary, and consider the backward Euler-Maruyama scheme
\be \label{i:BEM}
\what{X}_{t_{n+1}} =\what{X}_{t_n} +\sigma(\what{X}_{t_n})(B_{t_{n+1}}-B_{t_n}) +  \frac{T}{N}\sigma^2(\what{X}_{t_n})\frac{h'}{h}(\wXnp),
\ee
where $h$ becomes a concave function.

Note that different than what one would expect from a backward scheme (see, e.g. Mao and Szpruch \cite{MSbem}, Alfonsi \cite{AlfCIR1}, Alfonsi \cite{AlfCIR2}, and Neunkirch and Szpruch \cite{NSimp} to name a few) the $\sigma^2$-term in the drift of (\ref{i:Xrtr}) is still evaluated at $\wXn$. This stems from the fact that (\ref{i:Xrtr}) with $b\equiv 0$ should be viewed as a time-changed version of 
\[
dY_t=dW_t +\frac{h'}{h}(Y_t)dt,
\]
where the time change is given by $\int_0^t\sigma^2(Y_s)ds$. We make an extensive use of this correspondence in our proofs.

Our main result is Theorem \ref{t:main} which proves that the rate of weak convergence of the above backward Euler-Maruyama scheme is $N^{-1}$ under standard assumptions on the diffusion process. Moreover, there is no single $h$ function that achieves this rate. We show that any nonnegative concave $h$ vanishing at accessible boundaries can be used to obtain this convergence rate as long as it satisfies some mild growth conditions. Such functions are easy to construct and we study in Section \ref{s:numerics} the construction of some particular $h$-functions to compute approximate prices for barrier options in a Black-Scholes framework. Our numerical results are very promising and error terms very rapidly converge to $0$ even with a small number of iterations. Moreover, in the case of a particular local volatility model with double barriers, our method yields smaller error terms than the so-called {\em Brownian bridge method} when the number of discretisations is reasonably large. 

We are not the first to consider implicit schemes for studying diffusions with infinite lifetime and taking values in a strict subset of $\bbR$. Alfonsi in  \cite{AlfCIR1, AlfCIR2} and  Neunkirch and Szpruch \cite{NSimp} consider such scalar processes whose SDE representation is given by
\be \label{i:ASimp}
dY_t= dW_t + f(Y_t)dt,
\ee
and $f$ satisfying the conditions of a Feller test ensuring that $Y$ takes values in $(\ell,r)$ (see also \cite{DNScir} in the special case of Cox-Ingersoll-Ross (CIR) process).  \cite{AlfCIR2} and \cite{NSimp} show that a the drift implicit Euler scheme for $Y$ converge strongly with rate $N^{-1}$ if $f$ satisfies certain integrability conditions including
\be \label{i:condimp}
\bbE^{\bbQ}\left[\int_0^T(f'(Y_t))^2dt\right]<\infty.
\ee
However, this condition cannot be satisfied by $h$ that paves the way for recurrent transformation rendering $X$ recurrent and following (\ref{i:Xrtr}). Indeed, if the dynamics of $X$ are given by (\ref{i:Xrtr}) where $h$ is a function satisfying the condition of Theorem 3.2 in \cite{rectr}, $b\equiv 0$, and $\sigma\equiv 1$, then
\[
d\frac{1}{h}(X_t)=-\frac{h'}{h^2}(X_t)dW_t +dC_t,
\]
where $C$ is an adapted, continuous and increasing process. Note that $h$ is concave and
\[
\left(\frac{h'}{h}\right)'=\frac{h''}{h}-\left(\frac{h'}{h}\right)^2.
\] 
Moreover, $h'$ never vanishes at the boundary points where $h$ does. Thus, (\ref{i:condimp}) implies that the local martingale in the above $\bbQ$-Doob-Meyer decomposition of $\frac{1}{h(X)}$ is a true martingale, which in turn will yield $\frac{1}{h}\exp(-A)$ is a true martingale, where $dA_t=-\half \frac{h''(X_t)}{h(X_t)}$. But this would imply that $\bbP\sim \bbQ$ (when restricted to $\cF_t$, with $(\cF_t)$ representing the underlying filtration) in view of the absolute continuity relationship manifested in Theorem 3.2 in \cite{rectr}. This is not possible since for an arbitrary $t>0$ $\bbQ(\zeta<t)=0$ while $\bbP(\zeta<t)>0$. 

The estimates obtained by the authors in \cite{AlfCIR2} and \cite{NSimp} rely on the Burkholder-Davis-Gundy (BDG) inequality which requires the corresponding local martingale be a true martingale. As $\frac{1}{h(X)}$ is a strict local submartingale under $\bbQ$, one needs to develop new techniques to arrive at the needed estimate for convergence theorem. 

This brings to the fore another novelty of our paper. Given the impossibility of the use of BDG inequality we use potential theoretic methods that yield the boundedness of inverse moments of $h(X)$ under $\bbQ$, which is crucial for obtaining the weak convergence result in our paper (or a strong convergence type results considered by Alfonsi, Neunkirch and Szpruch). We use the theory of {\em Kato class} potentials to show the boundedness of required moments.  Kato potentials are one of the fundamental objects in the study of Schr\"odinger operators (see, e.g. \cite{AiSiSchr}, \cite{CFZGauge}, \cite{ChenGauge}, \cite{ChenSongGauge}). We show in Theorem \ref{t:fundtr} that the additive functional $dA_t=-\half \frac{h''(X_t)}{h(X_t)}$ belongs to a particular Kato class defined in \cite{ChenGauge}, which in turn yields the boundedness of the inverse moment of $\frac{1}{h}(\wXn)$ (uniformly in $N$) in conjunction with a comparison argument via Lemma \ref{l:tchange}. The potential theory also helps us to prove uniform bounds on the moments of integral functionals of $h^{-2-p}(\wXt)$ (see Theorem \ref{t:bemest} for an exact description).

Our methodology offers hope to study the convergence rates for CIR processes that do not satisfy (\ref{i:condimp}). We also show in this paper that if one considers the  $3$-dimensional Bessel process,
\[
dX_t=dW_t+ \frac{1}{X_t}dt,
\]
the implicit scheme in (\ref{i:BEM}) converges weakly at rate $N^{-1}$. Clearly, (\ref{i:condimp}) is violated since the reciprocal of a $3$-dimensional Bessel process is a prime example of a strict local martingale. This process satisfies the conditions of Theorem \ref{t:main} and one obtains the optimal convergence rate for sufficiently smooth $g$.  We leave the investigation of the convergence rate for conservative diffusions on $(0,\infty)$ satisfying (\ref{i:ASimp}) in the absence of condition (\ref{i:condimp}) to a future study.

Although our analysis assumes a one-dimensional framework,  a close look into our technical analysis reveals that our convergence result does not depend heavily on this assumption apart from the comparison argument used in Lemma \ref{l:tchange}. In particular it is relatively clear how to obtain a  version of Theorem \ref{t:invmomnt} in the multidimensional case using  well-known potential theoretic arguments on Kato classes.  However, our main obstacle in not being able to  extend our results to  a multidimensional setting is the absence of a systematic study of recurrent transformations in higher dimensions.  Also note that  Lemma \ref{l:tchange} is only used to obtain estimates on  $h(\what{X})$,  which is always a one-dimensional object with $\what{X}$ referring to the continuous Euler scheme.  Such a study and its applications to Euler methods for killed diffusions will be the subject of future research. 

The outline of the paper is as follows. Section \ref{s:prelim} fixes the setting, gives a brief summary of results for recurrent transformations needed for this paper together with novel inverse moment estimates, and introduces the backward Euler-Maruyama scheme that is tailored for our purposes. Section \ref{s:moment} obtains the moment estimates that will be needed for the weak convergence analysis. Theoretical predictions are confirmed via numerical studies in Section \ref{s:numerics} and Section \ref{s:conclusion} concludes.
\section{Preliminaries} \label{s:prelim}
Let $X$ be a regular  diffusion on $(\ell,r)$, where $ -\infty \leq \ell <r \leq \infty$.  We assume that infinite boundaries are inaccessible and if any of the boundaries are reached in finite time, the process is killed and sent to the cemetery state $\Delta$. This is the only instance when the process can be `killed', we do not allow killing inside $(\ell,r)$. The set of points that can be reached in finite time starting from the interior of $(\ell,r)$ and entrance boundaries will be denoted by $I$. That is, $I$ is the union of $(\ell,r)$ with the regular, exit and entrance boundaries.  The law induced on $C(\bbR_+,I)$,  the space of $I$-valued continuous functions on $[0,\infty)$, by $X$ with $X_0=x$ will be denoted by $P^x$ as usual, while $\zeta$ will correspond to its lifetime, i.e. $\zeta :=\inf\{t>0:X_{t} \notin (\ell,r)\}$. We also introduce the set $I_{\Delta}:=I\cup \{\Delta\}$ and extend any $I$-valued Borel measurable function $f$ to $I_{\Delta}$ by setting $f(\Delta)=0$ unless stated otherwise.  The filtration $(\cF^0_t)_{t \geq 0}$ will correspond to the natural filtration of $X$, $\tilde{F}_t$ will be the universal completion of $\cF^0_t$, and $\cF_t=\tilde{F}_{t+}$ so that $(\cF_t)_{t\geq 0}$ is a right continuous filtration.   We will also set $\cF:=\bigvee_{t \geq 0} \cF_t$.  We refer the reader to \cite{BorSal} for a summary of results and references on one-dimensional diffusions. The definitive treatment of such diffusions is, of course, contained in \cite{IM}. 

Since we are only concerned with the diffusion process until it is killed, we can assume without any loss of generality that $X$ is on natural scale. The extra regularity conditions imposed in the following assumption are standard in  the theory of Euler discretisations for SDEs.
\begin{assumption} \label{a:reg} $X$ is a regular one-dimensional diffusion on $(\ell,r)$ such that
	\[
	X_t=X_0+ \int_0^t\sigma(X_s)dB_s, \qquad t < \zeta,
	\]
	where   $\sigma:(\ell,r)\to (0,\infty)$ is continuously differentiable with a bounded derivative, $B$ is a  standard Brownian motion, and $\zeta=\inf\{t>0: X_{t}\in \{l,r\}\}$.  Moreover, $\sigma(\ell+)$  (resp. $\sigma(r-)$)  exist and is finite if $\ell$ (resp. $r$ )  is finite.
\end{assumption}
Note that the speed measure $m$ associated with $X$ is given by $m(dx)=2\sigma^{-2}(x)dx$ on the Borel subsets of $(\ell,r)$.

Since we are interested in diffusions with killing, the following assumption is needed to ensure that we are not dealing with a vacuous problem.

\begin{assumption}
	\label{a:killing} $P^x(\zeta<\infty)>0$ for each $x \in \elr$.
\end{assumption}

Let $I_0$ be the set of points in $I$ that can be reached from its interior in finite time. Note that under Assumptions \ref{a:reg} and \ref{a:killing} there are only two cases to consider:
\begin{itemize}
	\item[Case 1:] Both $\ell$ and $r$ are accessible, which in turn implies $\ell$ and $r$ are finite and $I_0=[\ell,r]$. 
	\item[Case 2:] Only one of $\ell$ and $r$ is accessible, which can be assumed to be $\ell$ without  any loss of generality. In particular, $I^0=[\ell, r)$. 
\end{itemize}
As $\ell$ is always finite as a result of the above convention,   the following will also be assumed for convenience:
\begin{assumption}
	\label{a:ell} $\ell=0$.
\end{assumption}
As a transient diffusion on $(\ell,r)$, $X$ has a finite potential density, $u: (\ell,r)^2 \to \bbR_+$, with respect to its speed measure (see Paragraph 11 in Section II.1 of \cite{BorSal}). That is, for any nonnegative and measurable  $f$ vanishing at accessible boundaries
\[
Uf(x):=\int_0^{\infty}E^x[f(X_t)]dt=\int_l^r f(y)u(x,y)m(dy).
\]
The potential density is symmetric and is explicitly known in terms of the scale function and the speed measure of $X$. This leads to the following specification of the potential density:
\be \label{e:potden}
u(x,y)=\left\{\ba{ll}
(x\wedge y)\left(1-\frac{x \vee y}{r}\right), & \mbox{ if } r<\infty,\\
	x\wedge y, &\mbox{ otherwise.}\ea \right.
\ee

The following is a direct consequence of Theorem 3.2 in \cite{rectr}. The reader is referred to \cite{rectr} for all unexplained terminology.
\begin{theorem} \label{t:fundtr} Suppose that Assumptions \ref{a:reg}-\ref{a:ell} are in force.  Let $f:(0,r)\to (0,\infty)$ be a continuous function such that  $\int_{(0,r)}f(y)m(dy)<\infty$  and   $\int_{(0,r)}yf(y)m(dy)<\infty$,  and define 
	\[
	h(x):=\int_{(\ell,r)} u(x,y)f(y)m(dy).
	\]
	Then, the following hold:
	
	\begin{enumerate}
		\item $(h,M)$ is a recurrent transform of $X$, where
		\[
		M_t :=\exp\left(\int_0^t\frac{f(X_s)}{h(X_s)}ds\right).
		\]
		\item There exists a probability measure  $Q^{h,x}$ on $\cF$ that is locally absolutely continuous with respect to $P^x$ such that
		\be \label{e:SDErt}
		dX_t=\sigma(X_t)dW_t +\sigma^2(X_t)\frac{h'(X_t)}{h(X_t)}dt
		\ee
		and $W$ is an $Q^{h,x}$-Brownian motion. 
		\item If  $S$ is a stopping time such that $Q^{h,x}(S<\infty)=1$, then for any $ F \in \cF_S$ the following identity holds:
		\[
		P^x(\zeta>S, F)= h(x)E^{h,x}\left[\chf_F \frac{1}{h(X_S)}\exp\left(-\int_0^S\frac{f(X_s)}{h(X_s)}ds\right)\right],
		\]
		where $E^{h,x}$ is the expectation operator with respect to the probability measure $Q^{h,x}$. In particular, $Q^{h,x}(\zeta<\infty)=0$.
	\end{enumerate}
\end{theorem}
Note that $h$ constructed above is a concave function that is twice continuously differentiable and satisfies on $(\ell,r)$

\begin{equation} \label{eqforh}
\half \sigma^2 h''=-f.
\end{equation}

The class of concave functions $h$ such that $h=Uf$ where $f$ is a continuous function satisfying the conditions of Theorem \ref{t:fundtr} will be denoted by $\cH_0$. 

 We shall also consider the following $h$-transformation when $r=\infty$:
\begin{theorem} \label{t:ht}
Suppose that Assumptions \ref{a:reg}-\ref{a:ell} are in force and $r=\infty$.  Let $h(x):=x$.
Then, the following hold:

\begin{enumerate}
	\item There exists a probability measure  $Q^{h,x}$ on $\cF$ that is locally absolutely continuous with respect to $P^{x}$ such that
	\be \label{e:SDEht}
	dX_t=\sigma(X_t)dW_t +\sigma^2(X_t)\frac{h'(X_t)}{h(X_t)}dt
	\ee
		and $W$ is a $Q^{h,x}$-Brownian motion. 
	\item If  $S$ is a stopping time such that $Q^{h,x}(S<\infty)=1$, then for any $ F \in \cF_S$ the following identity holds:
	\[
	P^x(\zeta>S, F)= h(x)E^{h,x}\left[\chf_F \frac{1}{h(X_S)}\right],
	\]
	where $E^{h,x}$ is the expectation operator with respect to the probability measure $Q^{h,x}$. In particular, $Q^{h,x}(\zeta<\infty)=0$.
\end{enumerate}
\end{theorem}
The above result is well-known and the reader is referred to Theorem 6.2 in \cite{EH} for a proof in a much more general setting. Note that the $h$-transform of Theorem \ref{t:ht} does not produce a recurrent diffusion. Indeed, $Q^{h,x}(\lim_{t \rar \infty}X_t=\infty)=1$ since the corresponding scale function is finite at $\infty$.

For ease of later reference define the set $\cH$ to be the union of $\cH_0$ and the set that contains only the identity function when $r=\infty$. If $r$ is finite, set $\cH=\cH_0$.
\begin{lemma}\label{l:key} Let $h\in \cH$.
	\begin{enumerate}
		\item For any given $z> 0$ consider the function $H$
		defined by
		\begin{equation} \label{Hdefinition}
		H(x)=x-z \frac{h'(x)}{h(x)}, \qquad x\in (0,r).
		\end{equation}
		$H$ is strictly increasing and $H((0,r))=\bbR$.
		\item $h$ is increasing if $r=\infty$. However, $h'$ is bounded.  In particular, for $h\in \cH_0$,
		 we have
		 \[
		 \begin{split}		 
		 h'(0)&=\left\{\ba{ll}
		 \int_0^{\infty}f(y)m(dy), &\mbox{ if } r=\infty,\\
		 \int_0^{\infty} \frac{r-y}{r}f(y)m(dy)>0, &\mbox{ otherwise.}\ea
		 \right. \\
		 h'(r)&=  \left\{\ba{ll}
		 0, &\mbox{ if } r=\infty,\\
		 -\frac{1}{r}\int_0^r yf(y)m(dy)<0, &\mbox{ otherwise.}\ea
		 \right.
		 \end{split}
		 \]
		\item For any $\alpha\geq 0$ and $h\in \cH_0$
		\be \label{e:kato}
		\int_{(0,r)}u(y,y) \frac{\alpha |h'(y)|-h''(y)}{h(y)}dy<\infty.
arxiarr		\ee		
	\end{enumerate} 
\end{lemma}
\begin{proof}
	\begin{enumerate}
		\item Since $h$ is concave, $H'(x)>1$, which shows the desired strict monotonicity. 
		
		If $r=\infty$ and $h(x)=x$, that $H((0,\infty))=\bbR$ is immediate. 
		
		Next, suppose $h\in \cH_0$. Then, the dominated convergence theorem implies that $h(0)=0$ as well as $h(r)=0$ if $r<\infty$ since the potential density vanishes at finite endpoints. Moreover, as $h$ is strictly concave and never vanishes in the interior of the state space, $h'(0)>0$. Thus, 
		\[
		\lim_{x \rar 0}\frac{h'(x)}{h(x)}=\infty.
		\]
		This proves the desired range for $H$ when $r=\infty$. Indeed, in this case $h$ is increasing, which in turn yields
		\[
		\frac{h'(x)}{h(x)}\leq \frac{h'(1)}{h(1)} \qquad x\geq 1.
		\]
		If $r<\infty$, similar considerations imply $h'(r)<0$, and therefore
		\[
		\lim_{x \rar r}\frac{h'(x)}{h(x)}=-\infty.
		\]
		This completes proof of the first assertion.
		\item If $r=\infty$ and $h(x)=x$, $h'(x)=1$ for all $x\geq 0$. If $h\in \cH_0$,
		\[
		h'(x)=\int_x^{\infty}f(y)m(dy),
		\]
		which is nonnegative and finite by the assumption on $f$. In particular, $h'(0)=\int_0^{\infty}f(y)m(dy)$ and $h'(\infty)=0$.
		
		If $r<\infty$,
		\[
		h(x)=\frac{r-x}{r}\int_0^x yf(y)m(dy)+x\int_x^r  \frac{r-y}{r}f(y)m(dy).
		\]
		Thus,
		\[
		h'(x)=\int_x^rf(y)m(dy)-\frac{1}{r}\int_0^r yf(y)m(dy).
		\]
		This yields the desired boundedness and the boundary levels for the derivatives.
		\item First suppose $r<\infty$. Since $h'$ does not vanish at the boundaries, $u(y,y)/h(y)$ is bounded. 
		Moreover,
		\[
		-\int_{(0,r)}h''(y)dy=\int_{(0,r)}f(y)m(dy).
		\]
			This proves the claim when $r<\infty$. Now suppose that $r=\infty$. Thus, $u(y,y)=y$ and
			\[
			-\int_{(0,\infty)}y\frac{h''(y)}{h(y)}dy=\int_{(0,1)}y\frac{f(y)}{h(y)}m(dy)+\int_{(1,\infty)}y\frac{f(y)}{h(y)}m(dy).
			\]
			The first integral on the right hand side is finite since $f$ is $m$-integrable and $y/h(y)$ is bounded on $[0,1]$ as $h'(0)>0$. The second integral is also finite since $h(\infty)>0$ and $\int yf(y) m(dy)<\infty$ by assumption.
	\end{enumerate}
\end{proof}
Let $g:I_0\to \bbR$ be a continuous function vanishing at accessible boundaries and  Then for $h\in \cH$ and a deterministic $T>0$ we have
\be\label{e:RNkey}
E^x\left[g(X_T)\chf_{[T<\zeta]}\right]=h(x) E^{h,x}\left[ \frac{g(X_T)}{h(X_T)}\exp\left(\half \int_0^T\frac{\sigma^2(X_s)h''(X_s)}{h(X_s)}ds\right)\right].
\ee

In order to approximate the expectation on the right side of (\ref{e:RNkey}) we shall use a backward Euler-Maruyama (BEM) scheme:

Let $N>1$ be an integer and define $t_n:=\frac{n}{N}T$ for $n=0, \ldots, N$. Set $\bar{X}_0=X_0$ and proceed inductively by setting,Then proceed inductively by setting
\be \label{e:BEM}
\what{X}_t=\what{X}_{t_n}+\sigma(\what{X}_{t_n})(W_t-W_{t_n})+ (t-t_n)\sigma^2(\what{X}_{t_n})\frac{h'(\what{X}_{t})}{h(\what{X}_{t})}
\ee
for $t \in (t_n,t_{n+1}]$ and $n=0, \ldots N-1$.

 Note that in view of Lemma \ref{l:key} the mapping $x\mapsto x- z\frac{h'}{h}(x)$ is one-to-one and onto for any given $z>0$. Thus, the above scheme is well-defined since $\sigma(x)>0$ for all $x\in (0,r)$.
			
 As we shall see in Section \ref{s:moment} the following type of diffusion processes on $(0,r)$  will play a crucial role:
\be \label{e:Ydef}
dY_{t}=dW_t + \left\{\frac{h'(Y_t)}{h(Y_t)}+c\right\}dt, \quad t <\zeta(Y)
\ee
where $\zeta(Y)$ denotes the first hitting time of $0$ or $r$. Note that $c=0$ corresponds to the recurrent transform defined above. 
\begin{theorem}\label{t:invmomnt} Suppose that Assumptions \ref{a:reg}-\ref{a:ell} are in force, $h\in \cH$, and $Y$ is a process defined by (\ref{e:Ydef}) with  $Y_0=X_0$. Assume further that $c\leq 0$   if $r=\infty$, and $c=0$ if $h(x)=x$ for all $x$. Then the following statements are valid:
	\begin{enumerate}
		\item $Q^{h,X_0}(\zeta(Y)=\infty)=1$.
		\item For any stopping time $S$  that is bounded $Q^{h,X_0}$-a.s. there exists a constant $K$ that does not depend on $X_0$ such that
		\[
		E^{h,X_0}\left[\frac{1}{h(Y_{S})}\right]<\frac{K}{h(X_0)}.
		\]	
		\item For any $t>0$ and $p\in[0,1)$
		\[
		E^{h,X_0}\left[\int_0^t\frac{1}{h^{2+p}(Y_{s})}ds\right]<\infty.
		\]
	\end{enumerate} 
\end{theorem}
\begin{proof}
	\begin{enumerate}
		\item First observe that a scale function and speed measure for $Y$ can be chosen as
		\[
		s_y(x)= \int_d^x\frac{e^{-2cy}}{h^2(y)}dy, \quad m_y(dx) = 2 h^2(x)\exp{2cx} dx,
		\]
		where $d\in (0,r)$.	Since $s_y(0)=-\infty$, $0$ is an inaccessible boundary for $Y$. By the same token, $r$ is also an inaccessible boundary when $s_y(r)=\infty$, which will be valid when $r<\infty$ or $c\leq 0$.
		\item 	Define $Z$ by 
		\[
		Z_t:= \frac{1}{h(Y_{t})}\exp\left(\half\int_{0}^{t} \frac{ 2  ch'(Y_s)+ h''(Y_s)}{h(Y_s)}ds\right)
		\]
		and note that $Z$ is a nonnegative $Q^{h,X_0}$-local martingale by a straightforward application of Ito's formula. By Theorem 62.19 in \cite{GTMP} there exists a probability measure $\tilde{P}$ such that 
		\[
		dY_t=d\beta_t + c dt, \quad t<\zeta(Y),
		\]
		where $\beta$ is a $\tilde{P}$-Brownian motion, and whenever $S$ is a stopping time that is finite $Q^{h,X_0}$-a.s., one has
		\[
		\begin{split}
		&E^{h,X_0}\left[\frac{1}{h(Y_{S})}\right]=\frac{1}{h(X_{0})}\tilde{E}\left[\chf_{[S<\zeta(Y)]}\exp\bigg(-\half\int_{0}^{S}\frac{2  ch'(Y_s)+h''(Y^N_s)}{h(Y_s)}ds\bigg)\right]\\
		&\leq \frac{1}{h(X_{0})}\tilde{E}\left[\chf_{[S<\zeta]}\exp\bigg(\half\int_{0}^{S}\frac{2  (ch'(Y_s))^--h''(Y_s)}{h(Y_s)}ds\bigg)\right],
		\end{split}
		\]
		where $x^-$ denotes the negative part of $x$ and we drop the dependency on $Y$ for $\zeta$ to ease the exposition.
		
		Suppose that $S< R, \, Q^{h,X_0}$, a.s. where $R$ is a deterministic constant, and note that $\tilde{P}(S\geq R, S<\zeta)=0$. Thus,
		\[
		\tilde{E}\left[\chf_{[S<\zeta]}\exp\bigg(\half\int_{0}^{S}\frac{2  (ch'(Y_s))^--h''(Y_s)}{h(Y_s)}ds\bigg)\right]\leq \tilde{E}\left[\exp\bigg(\half\int_{0}^{R\wedge \zeta}\frac{2  (ch'(Y_s))^--h''(Y_s)}{h(Y_s)}ds\bigg)\right]
		\]
		
		Let $\cW^{c,y}$ denote the law of the process $\tilde{Y}$ starting at $y$, where $d\tilde{Y}_t=d\beta_t  +c dt$ and gets killed at hitting $0$ or $r$. Thus,
		\[
		\tilde{E}\left[\exp\bigg(\half\int_{0}^{R\wedge \zeta}\frac{2  (ch'(Y_s))^--h''(Y_s)}{h(Y_s)}ds\bigg)\right]=\cW^{c,X_0}\left[\exp(C_{R})\right],
		\]
		where	$C$ is the positive continuous additive functional of $\tilde{Y}$ with $dC_t= \half  \frac{2  (ch'(\tilde{Y}_t))^--h''(\tilde{Y}_t)}{h(\tilde{Y}_t)}\chf_{[t<\tilde{\zeta}]}dt$.
		
		Note that the potential function $u_C$ of $C$ is given by
		\[
		u_C(x)=\cW^x[C_{\infty}]=\int_0^{r} v (x,y)\mu_C(y)\frac{d\tilde{m}}{dy},
		\]
		where $v$ is the potential density of $\tilde{Y}$,  $\mu_C(y)=\half \frac{(2ch'(y))^-- h''(y)}{h(y)}$, and $d\tilde{m}$ is the associated speed measure of $Y$.   Since a scale function and a speed measure of $\tilde{Y}$ can be chosen as
		\[
		\tilde{s}(x)=\frac{1-e^{-2cx}}{2c} \mbox{ and } \tilde{m}(dx)=2e^{2cx}dx,
		\]
		where $\tilde{s}(x)=x$ if $c=0$, we obtain for $x\leq y$
		\[
		v(x,y)= \frac{\tilde{s}(x)(\tilde{s}(r)-\tilde{s}(y))}{\tilde{s}(r)},
		\]
		with $\frac{\tilde{s}(r)-\tilde{s}(y)}{\tilde{s}(r)}$ being interpreted as $1$ if $\tilde{s}(r)=\infty$. 
		
		First observe that $v(x,y)=u(x,y)$ if $c=0$. On the other hand, if $r=\infty$ and $c<0$
		\be \label{e:potbd1}
		v(y,y)e^{2cy}=\frac{e^{2cy}-1}{2c}\leq y.
		\ee
		
		Similarly, for $r<\infty$
		\be \label{e:potbd3}
		v(y,y)e^{2cy}=\frac{e^{-2cr}}{2c(1-e^{-2cr})}(e^{2cy}-1)(e^{2c(r-y)}-1)\leq K(c,r) y(1-\frac{y}{r}).
		\ee
		Thus, 
		\be \label{e:vyy2}
		\int_0^{r} v (y,y)\mu_C(y)2e^{2cy}dy\leq K \int_0^{r}u(y,y)\frac{(2c h'(y))^-- h''(y)}{h(y)}dy<\infty
		\ee
		by another application of (\ref{e:kato}) due to the bounds obtained via (\ref{e:potbd1}) and (\ref{e:potbd3}), and the assumption on the choice of $c$ when $r=\infty$. 
		
		As
		\[
		u_C(x)\leq \int_0^{r} v (y,y)\mu_C(y)2e^{2cy}dy,
		\]
		we deduce  that $u_c$ is bounded. 
		
		Now consider a decreasing sequence $(D_n)$ of subsets of $(0,r)$ such that $D_n \rar \emptyset$. Since
		\[
		\int_0^{r}v(x,y)\chf_{D_n}(y)\mu_C(y)2e^{2cy}\leq \int_0^{r}v(y,y)\chf_{D_n}(y)\mu_C(y)2e^{2cy}dy,
		\]
		and the right side converges to $0$ by the dominated convergence theorem due to  (\ref{e:vyy2}),  we establish that  $\mu_C \in \mathbf{K}_1(\tilde{Y})$ (see Definition 2.2 in Chen \cite{ChenGauge}) by Proposition 2.4 in  \cite{ChenGauge}. Therefore, by  Proposition 2.3 in \cite{ChenSongGauge} we arrive at
		\[
		\sup_{y\in (0,r)}\cW^{c,y}\exp(C_t)\leq d_1 e^{d_2 t}
		\]
		for some constants $d_1$ and $d_2$. This proves the claim.
		\item Since the semigroup is self-dual with respect to the speed measure, for any nonnegative measurable $f$ we have
		\[
		\begin{split}
\int_0^r dy2h^2(y)e^{2cy}f(y)E^{h,y}\int_0^t\frac{e^{-s}\chf_{[Y_s\in D]}}{h^{2+p}(Y_s)}ds=\int_D dy2h^2(y)e^{2cy}\frac{1}{h^{2+p}(y)}E^{h,y}\int_0^te^{-s}f(Y_s)ds\\
\leq \int_D dy\frac{2e^{2cy}}{h^{p}(y)}E^{h,y}\int_0^t f(Y_s)ds,
		\end{split}
		\]
		where $D:=\{y: h(y)<1\wedge \half \|h\|_{\infty}\}$. 
		
		In particular, when $f(y)=q(\eps, y,y^*)$ for some $\eps>0$, where $q$ is the transition density of $Y$ with respect to its speed measure, we obtain 
	\[
	\begin{split}
\int_0^r dy2h^2(y)e^{2cy}q(\eps, y,y^*)E^{h,y}\int_0^t\frac{e^{-s}\chf_{[Y_s\in D]}}{h^{2+p}(Y_s)}ds\leq \int_D dy2e^{2cy}h^{-p}(y)E^{h,y}(L^{y*}_{t+\eps})\\
\leq E^{h,y^*}(L^{y*}_{t+\eps})\int_D dy2e^{2cy}h^{-p}(y),
	\end{split}
	\]	
	where $L^{y*}$ is the diffusion local time with respect to the speed measure. Letting $\eps\rar 0$ we arrive at
	\[
	E^{h,y^*}\int_0^t\frac{e^{-s}\chf_{[Y_s\in D]}}{h^{2+p}(Y_s)}ds\leq 	E^{h,y^*}(L^{y*}_{t})\int_D dy2e^{2cy}h^{-p}(y)<\infty,
	\]
	 provided $y\mapsto E^{h,y}\int_0^t\frac{e^{-s}\chf_{[Y_s\in D]}}{h^{2+p}(Y_s)}ds$ is lower semi-continuous. Note that the finiteness of the integral on the right hand side follows from the fact that  $|h'(y)|\geq \alpha$ for some $\alpha>0$ on $D$.
	
Observe that
	\[
	E^{h,y}\int_0^t\frac{e^{-s}\chf_{[Y_s\in D]}}{h^{2+p}(Y_s)}ds=\phi(y)-e^{-t}E^{h,y}(\phi(Y_t)),
	\]
	where
	\[
	\phi(y):=E^{h,y}\int_0^{\infty}ds \frac{e^{-s}\chf_{[Y_s\in D]}}{h^{2+p}(Y_s)}ds=\int_D\frac{2 e^{2cz}v_1(y,z)}{h^p(z)}dz,
	\]
	where $v_1$ is the $1$-potential density of $Y$. Since $v_1$ is jointly continuous (see Paragraphs 10-11 in Chapter II of \cite{BorSal}), the claimed semi-continuity follows.
	
	Since
	\[
	E^{h,y^*}\int_0^t\frac{1}{h^{2+p}(Y_s)}ds\leq e^t E^{h,y^*}\int_0^t\frac{e^{-s}\chf_{[Y_s\in D]}}{h^{2+p}(Y_s)}ds + K,
	\]
	for some $K$, the claim follows from the arbitrariness of $y^*$.

	\end{enumerate}
\end{proof}
\section{Moment estimates for the continuous BEM scheme}\label{s:moment}
In this section we will obtain some moment estimates, including inverse ones, that will be necessary to establish the weak rate of convergence. We start with the following consequence of Ito's formula. 
\begin{lemma}\label{l:dcomp}
	Suppose that $h\in C^2_b((0,r), (0,\infty))$, $h^{(3)}$ exists and satisfies $|h^{(3)}| \leq K(1+h^{-p})$ for some constant $K$ and $p\in [0,1)$. Consider the BEM scheme defined by (\ref{e:BEM}) for $h\in \cH$. Then
	\be \label{e:Xhdecomp}
	d\what{X}_t=\frac{\sigma(\wXn)}{H_x(t_n,\wXn;t,\wXt)}dW_t +\frac{\sigma^2(\wXn)}{H_x^2(t_n,\wXn;t,\wXt)}\left\{\frac{h'}{h}(\wXt) +\mu(t_n, \wXn; t, \wXt)\right\} dt, \quad t\in (t_n,t_{n+1}],
	\ee
	where
	\[
	\begin{split}
H(t_n,z;t,x)&:=x-\sigma^2(z)(t-t_n)\frac{h'}{h}(x)\\
 \mu(t_n,z;t,x)&:=(H_x(t_n,z;t,x)-1)\frac{h'}{h}(x) +\half \frac{\sigma^2(z)(t-t_n)}{H_x(t_n,z;t,x)}\left(\frac{h'}{h}\right)''(x).
	\end{split}
	\]
	Consider the sets $O_1:=\{x:h'(x)>0\}$ and $O_2:=\{x:h'(x)<0\}$. 
	 Then 
	 \[
	 \inf_{x\in O_1}\mu(t_n,z;t,x)\geq c_1 \mbox{ and }  \sup_{x\in O_2}\mu(t_n,z;t,x)\leq c_2 \
	 \]
	 for some constants $c_1\leq 0\leq c_2$ that do not depend on $t_n, t$ or $z$. In particular, $c_1=0$ when $h(x)=x$.
\end{lemma}
\begin{proof}
The decomposition (\ref{e:Xhdecomp}) follows from Ito's formula and straightforward  calculations regarding the derivatives of the inverse function. 

To prove the second assertion first observe that $H_x(t,x)-1=-\sigma^2(z)(t-t_n)\Big(\frac{h'}{h}\Big)'(x)\geq 0$, where we drop the dependency on $t_n$ and $z$ to ease the exposition. 

 Observe that
\be \label{e:murep}
\mu=-\frac{\sigma^2(z)(t-t_n)\Big(\frac{h'}{h}\Big)'}{H_x}\left(H_x\frac{h'}{h}- \frac{1}{2}\frac{\left(\frac{h'}{h}\right)''}{\left(\frac{h'}{h}\right)'}\right),
\ee
and that the claim follows immediately if $h(x)=x$ since the term in the parenthesis in (\ref{e:murep}) becomes nonnegative. Thus, it remains to show the assertion when  $h\in \cH_0$.

First consider the case $r=\infty$, and let $u:=\frac{h'}{h}$ and note that $\lim_{x\rar \infty} u'(x)=0$ by Lemma \ref{l:key}. Moreover, $|u'(x)|\leq K x^{-2}$ for some $K <\infty$, which in turn implies
\be \label{e:u2u1}
\lim_{x\rar \infty} \frac{\log (-u'(x))}{x}=0=\lim_{x\rar \infty} \frac{ u''(x)}{u'(x)},
\ee
where the second equality is an application if L'Hospital's rule.  Thus,
\[
- \frac{1}{2}\frac{\left(\frac{h'}{h}\right)''}{\left(\frac{h'}{h}\right)'}> c \mbox{ on } (\frac{x^*}{2},\infty)
	\]
	for some $c<0$ where $x^*:=\inf\{x:h'(x)=0\}>0$ by Lemma \ref{l:key}.

An alternative representation for $\mu$ is given by
\be \label{e:mualt}
\mu= \sigma^2(z)(t-t_n)\left(-\frac{h'}{h}\left(\frac{h'}{h}\right)'\frac{1+H_x}{H_x}+\frac{1}{2H_x}\frac{h'''h-h''h'}{h^2}\right).
\ee
Thus, we will be done if 
\[
r(t,x):=\frac{\sigma^2(z)(t-t_n)}{2H_x}\frac{h'''h-h''h'}{h^2}
\]
is bounded from below on $(0,\frac{x^*}{2})$. Indeed, as $h'$ is bounded away from $0$ on this interval, the hypothesis on $h'''$ implies
\[
r(t,x)\geq -K\frac{\sigma^2(z)(t-t_n)\left(\frac{h'}{h}\right)^2}{1+ \sigma^2(z)(t-t_n)(\frac{h'}{h})^2}
\]
leading to the desired lower bound.

When $r<\infty$, we have in particular that $\sigma$ is bounded. Moreover,
\[
|r(t,x)|\leq K\frac{\sigma^2(z)(t-t_n)\frac{1}{h^2}}{2H_x},
\]
for some constant $K$, which renders $r$ bounded.  Observing that the remaining terms in (\ref{e:mualt}) has the correct sign completes the proof.
\end{proof}
The next result is a key comparison result that relates the inverse moments of the BEM scheme to those of the process (\ref{e:Ydef}) and thereby provide estimates that are valid uniformly in $N$.
\begin{lemma}\label{l:tchange}
Suppose that $h$ satisfies the conditions of Lemma \ref{l:dcomp}, $\sigma$ is bounded, $r=\infty$, and  consider the BEM scheme defined by (\ref{e:BEM}) for $h\in \cH$. Then for any non-decreasing and measurable function $\phi$ that does not change sign, we have
\[
 E^{h,X_0}(\phi(\what{X}_{A_t^{-1}}))\geq  E^{h,X_0}(\phi(Y_{t}))
 \]
 where $Y$ is the process defined by (\ref{e:Ydef}) with $c=c_1$, $c_1$ is as in Lemma \ref{l:dcomp}, and $A$ is a continuous time-change defined by $A_0=0$ and
 \[
 dA_t=\frac{\sigma^2(\wXn)}{H_x^2(t_n,\wXn;t,\wXt)}dt,\qquad t\in (t_n,t_{n+1}].
 \]
 Moreover, $Q^{h,X_0}(A_t\leq t\|\sigma\|_{\infty}^2)=1$.
\end{lemma}
\begin{proof}
Consider the process $\what{Y}$ defined by $\what{Y}_t=\what{X}_{A_t^{-1}}$.

Dambis, Dubins-Schwarz Theorem (cf. Theorem V.1.6 in \cite{RY}) yields
\[
d\what{Y}_t=d\beta_t +\left(\frac{h'}{h}(\wYt) +\mu_t\right)dt, \quad t\in (t_n,t_{n+1}],
\]
where $\mu_t\geq c_1$ and $\beta$ is a standard Brownian motion adapted to the filtration $(\cF_{A_t^{-1}})_{t\geq 0}$.

Then  the comparison theorem for stochastic differential equations (cf. Theorem 2.10 in \cite{DMB-CD}) show that
\[
P^{h,X_0}(\wYt\geq Y_t, t\leq T)=1,
\]
where 
\be\label{e:Ydef2}
Y_t= X_0+  \beta_t +\int_0^t  \left(\frac{h'}{h}(Y_s)+c_1\right)ds.
\ee

Since $H_x \geq 1$, it follows that $A_t\leq \|\sigma\|_{\infty}^2 t$. This completes the proof.  
\end{proof}

The main moment estimates are collected in the following theorem.
\begin{theorem} \label{t:bemest}
	Suppose that $h$ satisfies the conditions of Lemma \ref{l:dcomp}, $\sigma$ is bounded,  and consider the BEM scheme defined by (\ref{e:BEM}). Then for any $T>0$ and $p\in[0,1)$, the following statements are valid:
	\begin{enumerate}
		\item For each $m\in \bbN$ 
		\be \label{e:integrable1}
		\sup_{t\leq T, N} E^{h,X_0}\left(\frac{1}{h}(\wXt)+\sum_{n=0}^{N-1}\int_{t_n}^{t_{n+1}}\frac{\sigma^2(\wXn)h^{-2-p}(\wXt)}{H_x^{2}(t_n,\wXn;t,\wXt)}dt+ |\wXt|^m  |\right)<\infty.
		\ee
\item  For each $n$
\be \label{e:ui}
\esssup_{\tau \in \cT_n}E^{h,X_0}\Big(\frac{1}{h}(\what{X}_{\tau})+ X^m_{\tau}\big| \cF_{t_n}\Big)<\infty,
\ee
where $m\geq 0$ is an integer and $\cT_n:=\{\tau: \tau \mbox{ is a stopping time such that } \tau \in [t_n, t_{n+1}], Q^{h,X_0}\mbox{-a.s.}.\}$.   
\item Suppose further that $p\leq \half$ and that $\frac{h''}{h^{1-p}}$ is bounded. Then for each $n \in \bbN$ and $m\geq 0$
\be \label{e:hinvctrl}
E^{h,X_0}\left(\sum_{n=0}^{N-1}\int_{t_n}^{t_{n+1}}\left(1-\exp\big((s-t_n)\sigma^2(\wXn)\frac{h''}{2h}(\wXn)\big)\right)\frac{\sigma(\wXn)^2(h^{-p}(\wXs)+\wXs^m)}{H_x^2(t_n,\wXn;s,\wXs)}ds\right)<\frac{KT}{N},
\ee
where $K$ is independent of $N$.
	\end{enumerate}
\end{theorem}

Proof of the above theorem is lengthy and is delegated to the Appendix. We end this section with the following lemma that will be useful in our PDE approach to weak convergence rate in the following section.
\begin{lemma} \label{l:weak}
	Suppose that $h$ satisfies the conditions of Lemma \ref{l:dcomp}, $\sigma$ is bounded,  and consider the BEM scheme defined by (\ref{e:BEM}).  Then for any $T>0$ the following statements are valid:
	\begin{enumerate}
		\item Let $p\in [0,1)$ and $m\geq 0$ be an integer. For each $n$ 
		\be 
		\begin{split}
		E^{h,X_0}\bigg(\int_{t_n}^{t_{n+1}}\bigg|\frac{h^{1-p}(\wXt)(1+\wXt^m)\mu(t_n, \wXn; t, \wXt)}{H_x^2(t_n,\wXn;t,\wXt)}\bigg|dt\Big| \cF_n\bigg) \\ \leq \frac{KT}{N}E^{h,X_0}\bigg(\int_{t_n}^{t_{n+1}}\frac{\sigma^2(\wXn)(h^{-2-p}(\wXt)+\wXt^m)}{H_x^2(t_n,\wXn;t,\wXt)}dt\big| \cF_n\bigg),
		\end{split}
	\ee
with $K$ being a constant independent of $n$.
	\item Assume further that $h\in C^4((0,r),(0,\infty))$. Consider $p\in [0,1)$ and suppose 
	\[
	\frac{|h^{(k)}|}{h}<\frac{K}{h^{k-2+p}}, \qquad k\in \{2,3,4\},
	\]
	 for some $K$. Let $f\in C^2((0,r), \bbR)$ be a bounded function such that
	\[
	|f^{(k)}(x)|\leq K(1+ x^m)h^{2-p-k}(x), \qquad k\in \{1,2\},
	\]
	for some $m\geq 0$.
	Then for each $n$ and $t\in [t_n,t_{n+1}]$ 
	\[
	\begin{split}
		\left|E^{h,X_0}\bigg(f(\wXt)\bigg\{\frac{h''}{h}(\wXn)-\frac{h''(\wXt)}{H_x^2(t_n,\wXn;t,\wXt)h(\wXt)}\bigg\}\Big| \cF_n\bigg)\right|\\
		\leq K E^{h,X_0}\bigg( \int_{t_n}^t\frac{\sigma(\wXn)^2(h^{-(2+p)}(\wXs)+\wXs^m)}{H_x^2(t_n,\wXn;s,\wXs)}ds\big| \cF_n\bigg)\\ -K\frac{h''}{h}(\wXn)E^{h,X_0}\bigg(\int_{t_n}^t\frac{\sigma(\wXn)^2\left((h^{-p}(\wXs)+\wXs^m)+(s-t_n)(h^{-2}(\wXs)+\wXs^m)\big)\right)}{H_x^2(t_n,\wXn;s,\wXs)}ds\big| \cF_n\bigg)\\
		\end{split}
		\]
		for some constant $K$ independent of $n$.  
		\item Suppose $f$ and $h$ satisfy the conditions of the previous part and $b\in C^2_b((0,r), \bbR)$. Then  for each $n$ and $t\in [t_n,t_{n+1}]$ ,
		\[
		\begin{split}
		\left|E^{h,X_0}\bigg(f(\wXt)\bigg\{b(\wXn)-\frac{b(\wXt)}{H_x^2(t_n,\wXn;t,\wXt)}\bigg\}dt\Big| \cF_n\bigg)\right|\\
		\leq K E^{h,X_0}\bigg( \int_{t_n}^t\frac{\sigma(\wXn)^2(h^{-2-p}(\wXs)+\wXs^m)}{H_x^2(t_n,\wXn;s,\wXs)}ds\big| \cF_n\bigg),
		\end{split}
		\]
		for some constant $K$ independent of $n$.  
	\end{enumerate}
\end{lemma}
\begin{proof}
	\begin{enumerate}
\item It follows directly from the definition of $\mu$ and the hypothesis on $h'''$ that
\[
h^{1-p}(\wXt)|\mu(t_n, \wXn; t, \wXt)|\leq K \sigma^2(\wXn)(t-t_n)h^{-2-p}(\wXt), \quad t\in [t_n,t_{n+1}],
\]
for some $K$. Also note that if $m\geq 1$ and $r=\infty$, there exists a $K$ such that $x^m h^{-(2+p)}\leq Kh^{m-(2+p)}$ for $x\in [0,1]$.  An analogous bound can be obtained near $r$ when $r$ is finite. Thus,
\be \label{e:loseprod}
x^m h^{-(2+p)}\leq K(x^m + h^{-(2+p)}).
\ee
\item Let $\mu_s:=\mu(t_n, \wXn; s, \wXs)$, $u:=\frac{h'}{h}$, and $\eta_s:= H_x(t_n,\wXn;s,\wXs)$. Then Ito's formula yields
\[
f(\wXt)\left(\frac{h''}{h}(\wXn)-\frac{h''(\wXt)}{h(\wXt)\eta_t^2}\right)=M_t + A_t,
\]
where $M$ is a local martingale with $M_{t_n}=0$ since $\eta_{t_n}=1$, and
\[
\begin{split}
A_t&=\int_{t_n}^t\frac{\sigma^2(\wXn)f(\wXs)}{2\eta_s^4}\left\{\frac{2h''h'}{h^2}(\wXs)\mu_s+\frac{(h'')^2-h h^{(4)}}{h^2}(\wXs)\right\}ds\\
&-\int_{t_n}^t\frac{\sigma(\wXn)^2f(\wXs)}{\eta_s^4}\frac{h^{(3)}}{h}(\wXs)\left\{\mu_s+2\sigma^2(\wXn)(s-t_n)\frac{u''(\wXs)}{\eta_s}\right\}ds\\
		&-\int_{t_n}^t\frac{\sigma^4(\wXn)(s-t_n)f(\wXs)h''(\wXs)}{h(\wXs)\eta_s^5}\left\{2\mu_s u''(\wXs)+\frac{3\sigma^2(\wXn)(s-t_n)(u'')^2(\wXs)}{\eta_s}+u^{(3)}(\wXs)\right\}ds\\
	&+\int_{t_n}^t\left(\frac{h''}{h}(\wXn)-\frac{h''(\wXs)}{h(\wXs)\eta_s^2}\right)\frac{\sigma^2(\wXn)}{\eta_s^2}\left\{f'(\wXs)(u(\wXs)+\mu_s)+\half f''(\wXs)\right\}ds\\
	&+\int_{t_n}^t\frac{\sigma^2(\wXn)f'(\wXs)}{\eta_s^4}\left\{\frac{h''h'-h h^{(3)}}{h^2}(\wXs)-2\frac{\sigma^2(\wXn)(s-t_n)u''(\wXs)}{\eta_s}\frac{h''}{h}(\wXs)\right\}ds.
\end{split}
\]
Observe that the hypothesis on $h$ implies that
\[
|u^{(k)}|\leq K h^{-1-k}, \qquad k \in \{0,1,2, 3\},
\]
for some constant $K$. Moreover, $|\mu_s|\leq K \sigma^2(\wXn) (s-t_n)h^{-3}$ and
\[
\sigma^2(\wXn)(s-t_n)h^{-2}\eta_s^{-1}\leq K,
\]
for some other constant $K$ that does not depend on $s$.

Thus,  combined with the assumption on $f$ we arrive at
\be\label{e:Abd}
\begin{split}
|A_t|\leq-K\frac{h''}{h}(\wXn)\int_{t_n}^t\frac{\sigma(\wXn)^2(1+\wXs^m)}{H_x^2(t_n,\wXn;s,\wXs)}\left(h^{-p}(\wXs)(1+(s-t_n)h^{-2}(\wXs))\right)ds\\
+K \int_{t_n}^t\frac{\sigma^2(\wXn)(1+\wXs^m)}{H_x^2(t_n,\wXn;s,\wXs)h^{2+p}(\wXs)}ds
\end{split}
\ee
for some constant $K$. This in particular implies  $M$  is  a martingale since we can deduce from the estimates (\ref{e:integrable1}) and (\ref{e:ui}) that  the set $\{f(\what{X}_{\tau})\frac{h''(\what{X}_{\tau})}{h(\what{X}_{\tau})\eta_{\tau}^2}:\tau \in (t_n, t_{n+1}] \mbox{ is a stopping time}\}$ is uniformly integrable as soon as we once again recall that $|h''/h| <Kh^{-p}$ for some $p<1$. Hence, the claim holds in view of (\ref{e:loseprod}).
\item Applying Ito's formula and repeating the similar estimates yields the claim.
\end{enumerate}
\end{proof}
\section{Weak convergence of the BEM scheme} \label{sec:Weak convergence of the BEM scheme}
 Consider the following stochastic differential equation on a filtered probability space $(\Om, \cG, (\cG_t)_{t\in [0,T]}, \bbP)$ satisfying the usual conditions:
 \be
 X_t=X_0+ \int_0^t\sigma(X_s)dW_s +\int_0^t \mu(X_s)ds
 \ee
 where $X_0\in (0,r)$, $\sigma$ and $\mu$ are bounded and Lipschitz on $(0,r)$, and $\sigma(x)>\eps$ for all $x\in (0,r)$. Let $\tau:=\inf\{t\geq 0: X_t\notin (0,r)\}$ for some $\eps>0$. We are interested in a numerical approximation for
 \[
 \bbE[\tilde{g}(X_T)\chf_{[T<\tau]}],
 \]
 for a sufficiently regular $\tilde{g}$. 
 
 Observe that by a Girsanov transformation we can rewrite the above expression in terms of a diffusion process satisfying the conditions in earlier sections. Indeed, defining $\bbQ$ on $\cG$ via
 \[
 \frac{d\bbQ}{d\bbP}=\exp\left(-\int_0^T\frac{\mu(X_s)}{\sigma(X_s)}dW_s-\half\int_0^T\frac{\mu^2(X_s)}{\sigma^2(X_s)}ds\right)
\]
renders $X$ solve
\[
dX_t= \sigma(X_s)dB_s,
\]
for a $\bbQ$-Brownian motion $B$. Therefore,
\be
\begin{split}
 \bbE[\tilde{g}(X_T)\chf_{[T<\tau]}]&=\exp(-F(X_0))\bbE^{\bbQ}\left[g(X_T)\exp\left(\int_0^T \sigma^2(X_t) b(X_t)dt\right)\chf_{[T<\tau]}\right], \mbox{ where }\\
 g(x)&=\tilde{g}(x)\exp(F(x)), \\
 F(x)&=\int_c^x \frac{\mu(y)}{\sigma^2(y)}dy, \\
 b&=-\half\bigg\{\big(\frac{\mu}{\sigma^2}\big)'+\frac{\mu^2}{\sigma^4}\bigg\},
\end{split}
\ee
and $c\in (0,r)$.

Thus, we may assume $\mu\equiv 0$ and consider
\be \label{e:girsanovK}
\begin{split}
&E^{X_0}\left[g(X_T)\exp\left(\int_0^T \sigma^2(X_t) b(X_t)dt\right)\chf_{[T<\zeta]}\right]\\&=h(x)E^{h,X_0}\left[\frac{g(X_T)}{h(X_T)}\exp\left(\int_0^T \sigma^2(X_t)\Big\{ b(X_t)+\frac{h''(X_t)}{2h(X_t)}\Big\}dt\right)\right],
\end{split}
\ee
where $X$ is a process satisfying Assumption \ref{a:reg}, $b$ is bounded, $\eps<\sigma< K_{\sigma}$ and $g$ is sufficiently regular. 
\begin{proposition} Suppose $b \in C^4_b((0,r), \bbR)$, $\sigma \in C^4_b((0,r)$, $h\in \cH$ with 
	\[
	\frac{|h^{(k)}|}{h}<\frac{K_h}{h^{k-2+p}}, \qquad k\in \{2,3,4\},
	\]
	for some $K_h$ and $p\in (0,1)$, $g\in C^6_b((0,r),\bbR)$ is a bounded function with  $g^{(k)}(0)=0$ (and $g^{(k)}(r)=0$ if $r <\infty$) for $k\in \{0,1,2,3,4\}$, and define for $t\leq T$
	\be \label{e:defv}
	v(T-t,x):=E^{h,x}\left[\frac{g(X_t)}{h(X_t)}\exp\left(\int_0^t \sigma^2(X_s)\Big\{ b(X_s)+\frac{h''(X_s)}{2h(X_s)}\Big\}ds\right)\right].
	\ee
	Then
	\be \label{e:pdev}
	v_t+\frac{\sigma^2}{2}v_{xx} +\sigma^2 \frac{h'}{h}v_x=-\sigma^2v\Big(b+\frac{h''}{2h}\Big).
	\ee
	Moreover, $v$ and $v_t$ are uniformly bounded and there exists a constant $K$  such that
	\be
	\sup_{t\leq T}\Big|\frac{\partial^k}{\partial x^k}v_t(t,x)\Big|+\sup_{t\leq T}\Big|\frac{\partial^k}{\partial x^k}v(t,x)\Big|\leq Kh^{2-p-k}(x), \qquad k\in \{1,2\}.
	\ee
\end{proposition}
\begin{proof}
	Note that $v(T-t,x)=\frac{u(T-t,x)}{h(x)}$, where
	\[
	u(T-t,x):=E^{x}\left[g(X_t)\exp\left(\int_0^t \sigma^2(X_s) b(X_s)ds\right)\chf_{[t<\zeta]}\right].
	\]
	Note that $u(t,0)=0$ for $t\leq T$. Moreover, it follows from Theorem 5.2 in \cite{Ladyzhenskaya} that $u$ is the unique solution of 
	\be \label{e:pdeu}
	u_t +\half \sigma ^2 u_{xx} + \sigma^2 u b=0,
	\ee
	and that
	\be \label{e:uderbound}
	\sup_{t\leq T, x\in (0,r)}\left|\frac{\partial^l}{\partial x^l}\frac{\partial^k}{\partial t^k}u\right|<\infty, \quad 0\leq 2k+l\leq 5.
	\ee
	
	Also note that since 
	\[
	\begin{split}
	&E^{x}\left[g(X_t)\exp\left(\int_0^t \sigma^2(X_s) b(X_s)ds\right)\chf_{[t<\zeta]}\right]=g(x)\\
	&+\half E^{x}\left[\int_0^t\sigma^2(X_u)\exp\left(\int_0^u \sigma^2(X_s) b(X_s)ds\right)(g''(X_u)+2g(X_u)b(X_u))\chf_{[u<\zeta]}du\right],
	\end{split}
	\]
	we have
	\be \label{e:u_t}
	u_t(t,x)=-\half E^{x}\left[\sigma^2(X_{T-t})\exp\left(\int_0^{T-t} \sigma^2(X_s) b(X_s)ds\right)(g''(X_{T-t})+2g(X_{T-t})b(X_{T-t}))\chf_{[u<\zeta]}\right].
	\ee
	In particular, $u_t(\cdot,0)=0$, which in turn implies $u_{xx}(\cdot,0)=0$. Analogous boundary conditions also holds at $r$ if $r$ is finite.
	
	Let $w:=u_t$ and note that $w$ solves (\ref{e:pdeu}) with the boundary condition $w(t,0)=0$ and $w(T,\cdot)=-\half \sigma^2 g'' -\sigma^2 gb$. Using the stochastic representation in (\ref{e:u_t}) and  analogous arguments we again arrive at $w_t$ vanishing at finite boundaries .

	Using the PDE for $u$ it is straightforward to establish that $v$ solves (\ref{e:pdev}) and is bounded. Moreover, as $v_x=\frac{h u_x- uh'}{h^2}$, using integration by parts we arrive at
	\[
	v_x(t,x)=\frac{\int_0^x\left\{h(y)u_{xx}(t,y)-u(t,y)h''(y)\right\}dy}{h^2(x)}
	\]
	Since $h'(0)<\infty$ and $u$ and $u_{xx}$ vanish at $0$ (and are jointly continuous near $t=T$), there exists a neighbourhood of $0$ in which $|h''|(y) \leq Kh^{1-p}(y)\leq K^2 y$, $|u(\cdot,y)|+|u_{xx}(\cdot,y)|<Ky$ (due to Lipschitz continuity), and $h(y)>cy$. Thus, whenever $x$ belongs to this neighbourhood, we have
	\[
	\frac{v_x(t,x)}{h^{1-p}(x)}\leq \frac{K\int_0^x\left\{y(Ky+K^2 y^{1-p})\right\}dy}{c^{3-p}x^{3-p}}=\frac{K^2/3 x^3+K^3/(3-p)x^{3-p}}{c^{3-p}x^{3-p}}.
	\]
	Thus, $v_x/h^{1-p}$ is bounded near $0$. Analogous considerations when $r<\infty$ shows that the ratio is bounded over $(0,r)$.
	
	Next observe that $v_t$ is bounded since $u_t$ vanishes at finite boundaries and $u_{tx}$ is bounded. In particular, $v_t h^p$ remain bounded near finite boundaries (uniformly in $t$). Multiplying (\ref{e:pdev}) by $h^p$ and using the fact that $v_x/h^{1-p}$ is bounded demonstrate that 
	\[
	\sup_{t\leq T, x\in (0,r)}|v_{xx}(t,x)h^p(x)|<\infty.
	\] 
	
	Finally, since $v_t=\frac{w}{h}$, repeating the above arguments and using the fact that $w_{xx}$ vanish at finite boundaries and is Lipschitz  continuous in view of (\ref{e:uderbound}), we deduce $v_{tx}/h^{1-p}$ is bounded. Similar arguments (due to the boundedness of $w_{tx}=u_{ttx}$ in view of (\ref{e:uderbound}) also lead to
	\[
	\sup_{t\leq T, x\in (0,r)}|v_{txx}(t,x)h^p(x)|<\infty.
	\] 
\end{proof}
In view of the above  proposition, and for  the convenience of the reader, we collect all the assumptions needed in Assumption \ref{a:weak} below to prove our convergence result.
\begin{assumption} \label{a:weak} The functions $\sigma, b, h$ and $g$ satisfy the following regularity conditions.
\begin{enumerate}
	\item $h\in \cH\cap C^4((0,r), (0,\infty))$ such that 
	\[
	\frac{|h^{(k)}|}{h}<\frac{K_h}{h^{p+k-2}}, \qquad k\in \{2,3,4\},
	\]
	for some $K_h$ and $p\in [0,\half]$. 
	\item $\sigma \in C_b^2((0,r),(0,\infty))$ is bounded away from $0$.
	\item $b \in C^2_b((0,r), \bbR)$.
	\item $g\in C((0,r),\bbR)$  is of polynomial growth with $g(0)=0$  (and $g(r)=0$ if $r <\infty$).
	\item The function $v$ defined by (\ref{e:defv}) belongs to $C^{1,4}((0,r),\bbR)$, satisfies (\ref{e:pdev}) as well as the growth conditions
	\[
	\sup_{t\leq T}\Big|\frac{\partial^k}{\partial x^k}v_t(t,x)\Big|+\sup_{t\leq T}\Big|\frac{\partial^k}{\partial x^k}v(t,x)\Big|\leq K(1+x^m)h^{2-p-k}(x), \qquad k\in \{1,2\},
	\]
	for some constant $K$ and integer $m\geq 0$.
\end{enumerate}
\end{assumption}
\begin{remark}  The first condition on the derivatives of $h$ is not restrictive for practical purposes. Indeed, if a given $h\in \cH\cap C^4((0,r), (0,\infty))$ does not satisfy this condition, one can always linearise this concave function near the boundaries at which $h$ vanishes to obtain a new concave function satisfying the stated condition.
\end{remark}
\begin{theorem} \label{t:main}
Consider the BEM scheme defined by (\ref{e:BEM}) as well as the associated error 
\[
e(N):=\frac{g(X_T)}{h(X_T)}\exp\left(\int_0^T \sigma^2(X_t)\Big\{ b(X_t)+\frac{h''(X_t)}{2h(X_t)}\Big\}dt\right)-\frac{g(\what{X}_T)}{h(\what{X}_T)}\exp\left(\sum_{n=0}^{N-1}\frac{T}{N} \sigma^2(\wXn)\Big\{ b(\wXn)+\frac{h''(\wXn)}{2h(\wXn)}\Big\}\right).
\]
Then
\[
\big|E^{h,X_0}[e(N)]\big|\leq \frac{KT}{N},
\]
for some constant $K$ independent of $N$ under Assumption \ref{a:weak}.
\end{theorem}
\begin{proof}
	Let $\pi_0(s)= 1$, 
	\[
	\pi_k(s):=\exp\left(\sum_{n=0}^{k-1}s \sigma^2(\wXn)\Big\{ b(\wXn)+\frac{h''(\wXn)}{2h(\wXn)}\Big\}\right), k=1,\ldots, N,
	\]
	with the convention that $\pi_k =\pi_k(TN^{-1})$, and observe that
	\[
	\begin{split}
E^{h,X_0}[e(N)]&=E^{h,X_0}\left[v(T,\what{X}_T)\pi_N\right]-v(0,X_0)\\
&=\sum_{n=0}^{N-1}E^{h,X_0}\left[v(t_{n+1},\wXnp)\pi_{n+1}-v(t_{n},\wXn)\pi_{n}\right]\\
&=\sum_{n=0}^{N-1}E^{h,X_0}\left[\pi_{n}\Big(v(t_{n+1},\wXnp)\exp\Big(TN^{-1} \sigma^2(\wXn)\Big\{ b(\wXn)+\frac{h''(\wXn)}{2h(\wXn)}\Big\}\Big)-v(t_{n},\wXn)\Big)\right]
	\end{split}
\]
Next observe that
\[\begin{split}
E^{h,X_0}\left[\pi_{n}\Big(v(t_{n+1},\wXnp)\exp\Big(TN^{-1} \sigma^2(\wXn)\Big\{ b(\wXn)+\frac{h''(\wXn)}{2h(\wXn)}\Big\}\Big)-v(t_{n},\wXn)\Big)\Big|\cF_n\right]\\
=\pi_nE^{h,X_0}\left[\Big(v(t_{n+1},\wXnp)\exp\Big(TN^{-1} \sigma^2(\wXn)\Big\{ b(\wXn)+\frac{h''(\wXn)}{2h(\wXn)}\Big\}\Big)-v(t_{n},\wXn)\Big)\Big|\cF_n\right].
\end{split}
\]
Moreover, in view of (\ref{e:pdev}) (in fact dividing both sides of the equality by $\sigma^2$) we have
\[
v(t_{n+1},\wXnp)\exp\Big(TN^{-1} \sigma^2(\wXn)\Big\{ b(\wXn)+\frac{h''(\wXn)}{2h(\wXn)}\Big\}\Big)-v(t_{n},\wXn)=M_{t_{n+1}}-M_{t_n}+ I_1+ I_2 +I_3,
\]
where $M$ is a local martingale and
\[
\begin{split}
I_1&=\int_{t_n}^{t_{n+1}}\frac{\pi_{n+1}(t-t_n)}{\pi_{n}(t-t_n)}\frac{\sigma^2(\wXn)v_x(t,\wXt)\mu(t_n, \wXn; t, \wXt)}{H_x^2(t_n,\wXn;t,\wXt)}dt\\
I_2&=\int_{t_n}^{t_{n+1}}\frac{\pi_{n+1}(t-t_n)}{\pi_{n}(t-t_n)}\sigma^2(\wXn)v_t(t,\wXt)\Big(\frac{1}{\sigma^2(\wXn)}-\frac{1}{\sigma^2(\wXt)H_x^2(t_n,\wXn;t,\wXt)}\Big)dt\\
I_3&=\int_{t_n}^{t_{n+1}}\frac{\pi_{n+1}(t-t_n)}{\pi_{n}(t-t_n)}\sigma^2(\wXn)v(t,\wXt)\Big(b(\wXn)+\frac{h''(\wXn)}{2h(\wXn)}-\Big(b(\wXt)+\frac{h''(\wXt)}{2h(\wXt)}\Big)\frac{1}{H_x^2(t_n,\wXn;t,\wXt)}\Big)dt.
\end{split} 
\]
First note that $M$ is martingale due to (\ref{e:integrable1}) by the hypothesis on $v$ and that $h$ is bounded. Moreover, Lemma \ref{l:weak} shows (for a generic constant $K$ that may change from line to line albeit remaining bounded uniformly in $N$) such that
\[
\begin{split}
&\Big|E^{h,X_0}[I_1+I_2+I_3|\cF_n]\Big|\leq K\frac{T}{N} E^{h,X_0}\bigg( \int_{t_n}^{t_{n+1}}\frac{\sigma^2(\wXn)(h^{-2-p}(\wXs)+\wXs^m)}{H_x^2(t_n,\wXn;s,\wXs)}ds\big| \cF_n\bigg)\\
&+K E^{h,X_0}\bigg(\int_{t_n}^{t_{n+1}}dt\frac{\pi_{n+1}(t-t_n)}{\pi_{n}(t-t_n)}\sigma^2(\wXn)\int_{t_n}^t\frac{\sigma(\wXn)^2(h^{-2-p}(\wXs)+\wXs^m)}{H_x^2(t_n,\wXn;s,\wXs)}ds\big| \cF_n\bigg)\\
&-KE^{h,X_0}\bigg(\int_{t_n}^{t_{n+1}}dt\frac{\pi_{n+1}(t-t_n)}{\pi_{n}(t-t_n)}\frac{h''}{h}(\wXn)\sigma^2(\wXn)\bigg(\int_{t_n}^t\frac{\sigma(\wXn)^2(h^{-p}(\wXs)+\wXs^m)}{H_x^2(t_n,\wXn;s,\wXs)}ds\big| \cF_n\bigg)\\
&-KE^{h,X_0}\bigg(\int_{t_n}^{t_{n+1}}dt\frac{\pi_{n+1}(t-t_n)}{\pi_{n}(t-t_n)}\frac{h''}{h}(\wXn)\sigma^2(\wXn)\bigg(\int_{t_n}^t\frac{\sigma(\wXn)^2(h^{-2}(\wXs)+\wXs^m)(s-t_n)}{H_x^2(t_n,\wXn;s,\wXs)}ds\big| \cF_n\bigg)\\
&\leq K\frac{T}{N} E^{h,X_0}\bigg( \int_{t_n}^{t_{n+1}}\frac{\sigma^2(\wXn)(h^{-2-p}(\wXs)+\wXs^m)}{H_x^2(t_n,\wXn;s,\wXs)}ds\big| \cF_n\bigg)\\
&+E^{h,X_0}\bigg(\int_{t_n}^{t_{n+1}}\left(1-\exp\big((s-t_n)\sigma^2(\wXn)\frac{h''}{2h}(\wXn)\big)\right)\frac{\sigma(\wXn)^2(h^{-p}(\wXs)+\wXs^m)}{H_x^2(t_n,\wXn;s,\wXs)}ds\Big| \cF_n\bigg)\\
&+K\frac{T}{N}E^{h,X_0}\bigg(\int_{t_n}^{t_{n+1}}\frac{\sigma(\wXn)^2(h^{-2}(\wXs)+\wXs^m)}{H_x^2(t_n,\wXn;s,\wXs)}ds\Big| \cF_n\bigg),
\end{split}
\]
where we have used the boundedness of $\pi_{n+1}/\pi_n$ several times and the last two lines follow from the interchange of the order of integration on the third and the fourth lines.

This proves the assertion in view of Theorem \ref{t:bemest} and, in particular  (\ref{e:integrable1}) and (\ref{e:hinvctrl}), since $(\pi_n)$s are non-negative and uniformly bounded, and $H_x\geq 1$.
\end{proof}

\section{Numerical analysis}\label{s:numerics}

This section is dedicated to the numerical experiments illustrating the above technical analysis. As we shall see, one does not really need to satisfy all the conditions assumed in Theorem \ref{t:main} in order to achieve the advertised convergence rate in practice.  The experiments below will compare the our methodology developed in this paper to standard numerical approaches for pricing of barrier options. 

We shall consider the classical Black-Scholes model in the first part. As barrier option values are quite sensitive to the market skew/smile of volatility, the time-homogeneous {\it{hyperbolic local volatility}} model will also be studied and the corresponding results will be reported in the second part. 

\subsection{Black-Scholes model for barrier options} \label{subsec:BSBarrierOptions}

For expository purposes\footnote{Deterministic interest rate, dividend yield or borrow cost can be incorporated without difficulty.}, let's assume that the asset price follows

\begin{equation} \label{BSmodel}
\frac{dS_t}{S_t} = \sigma dW_t, \,\,\,\, S_0=1
\end{equation}
with  volatility $\sigma > 0$ under risk neutral probability $\bbP$. The value of the barrier option with payoff $\tilde{g}$ is given by
\begin{equation} \label{PriceEq}
price = \bbE^{\bbP} \left[ \tilde{g}(S_T) \mathds{1}_{\zeta > T} \right]
\end{equation}
with $\zeta :=\inf\{t>0:S_{t} \notin (e^{\ell},e^r)\}$, where $e^{\ell}$ represents the down barrier and $e^{r}$ the up barrier for $-\infty \leq \ell < r \leq \infty$.

For the volatility to be bounded away from $0$ (cf. point (2) in Assumption \ref{a:weak}), we perform a change of variable $X_t = \ln(S_t)$. Equations (\ref{BSmodel}) and (\ref{PriceEq}) then become
\begin{equation} \label{BSmodelinLog}
dX_t = -\frac{1}{2} \sigma^2dt + \sigma dW_t, \,\,\,\, X_0=x = 0
\end{equation}
\begin{equation} \label{PriceEqinLog}
price = \bbE^{\bbP} \left[ \tilde{g}(X_T) \mathds{1}_{\zeta > T} \right]
\end{equation}
with $\zeta =\inf\{t>0:X_{t} \notin  (\ell,r)\}$ and $\tilde{g}(x)$ still denotes the payoff function in $x$ variable by an abuse of notation.

To remove the drift in (\ref{BSmodelinLog}), we follow the Girsanov transformation described at beginning of Section \ref{sec:Weak convergence of the BEM scheme}. We thus obtain 
 \begin{equation} \label{BSmodelinLogwithGirsanov}
dX_t = \sigma dW_t, \,\,\,\, X_0=x
\end{equation}
under $\bbQ$, where $\frac{d \bbQ}{ d \bbP } = e^{ -\frac{1}{8} \sigma^2 T + \frac{1}{2} \sigma W_T }$. Consequently,
\begin{equation} \label{PriceEqinLogwithGirsanov}
price = e^{ \frac{1}{2} x - \frac{1}{8} \sigma^2 T} \bbE^{\bbQ} \left[ g(X_T) \mathds{1}_{\zeta > T} \right],
\end{equation}
and $g(x) = \tilde{g}(x) e^{-\frac{1}{2}x}$.\\

We shall perform a path transformation method described in earlier section that either produces a recurrent process (see Theorem \ref{t:fundtr}) or generates a transient process with infinite lifetime (see Theorem \ref{t:ht}).

\subsubsection{Specification of the recurrent transformation} \label{subsubsection:Recurrent h-transform}

\begin{itemize}

\item Double barrier case with $l$ and $r$ finite.

We shall pick
\begin{equation}\label{fdoublebarrier}\nn
f(x) = (x-l)(r-x)
\end{equation}
to construct the function $h$ via  (\ref{eqforh}). This in particular yields $h^{(3)}$ is bounded in $(\ell,r)$, which  in turn implies the boundedness of $\frac{h^{(2)}}{h}$ by means of L'Hopital's rule. In particular,  (1) in Assumption \ref{a:weak} is satisfied.  Double integration from (\ref{eqforh}) gives

\begin{equation} \label{hDoubleBarrier}\nn
h(x) = -\frac{2}{\sigma^2} \left[ -\frac{1}{12}x^4 + \frac{(l+r)}{6} x^3 - \frac{lr}{2} x^2 + ax \right] + b
\end{equation}
where $a$ and $b$  are given by
\begin{align*}
a &=  \frac{1}{(r-l)} \left[ \frac{(r^4 - l^4)}{12} - \frac{(r+l)}{6}(r^3 - l^3) + \frac{lr}{2}(r^2 - l^2)  \right], \nonumber \\
b &= \frac{1}{\sigma^2} \left[ -\frac{(r^4 + l^4)}{12} + \frac{(r+l)(r^3 + l^3)}{6} - \frac{rl}{2}(r^2 + l^2) + a(r+l)  \right].
\end{align*}

\item Single barrier case with $l$ finite and $r = + \infty$

We shall choose

 \begin{equation} \label{hforSinglebarrier}\nn
h(x) = e^{-l} - e^{-x} 
\end{equation}

with $h'(x) = e^{-x}$ and $h''(x) = -e^{-x}$. Note that with this choice of $h$ (1) in Assumption \ref{a:weak} is only partially satisfied as $\frac{|h''(x)|}{h(x)}$ is unbounded for $x$ around $l$.\\

\end{itemize}

We will apply  the implicit scheme (\ref{e:BEM})  so that the price (\ref{PriceEqinLog}) is  approximated by
\begin{equation} \label{BEMPrice}\nn
price \approx e^{ \frac{1}{2} x - \frac{1}{8} \sigma^2 T} h(x) \bbE^{h,x} \left[ \frac{g}{h}(\what{X}_{t_{N}}) e^{ \frac{\sigma^2}{2} \frac{T}{N} \sum_{n=0}^{N-1} \frac{h''}{h}( \what{X}_{t_{n}} ) }  \right]
\end{equation}

{\remark{In the Black-Scholes model with the change of variable $X_t = \ln(S_t)$, the H function is identical at each time step and needs to be computed once. In the implementation, we introduce a dense grid covering the interval $(l, r)$, calculate the values of H on these points and $H^{-1}$ is computed by piecewise constant approximation.}}

\subsubsection{The transient transformation} \label{subsubsection:Non-recurrent h-transform}

In the single barrier case of a down-and-out option that will constitute a part of our experiments we can also consider transformation via $h(x) = x-l$ when $l$ is finite and $r = +\infty$, as in Theorem \ref{t:ht}. Under $Q^{h,x}$, the the process $X$ defined in (\ref{BSmodelinLogwithGirsanov}) follows
\begin{equation} \label{BSmodelinLogwithhId}\nn
dX_t = \sigma dW_t + \frac{\sigma^2}{X_t - l}dt, \,\,\,\, X_0=x.
\end{equation}

One advantage of this transformation is that the inverse of the function $H$ appearing in the implicit scheme (\ref{e:BEM}) can be computed analytically and is given by
 \begin{equation}\nn
H^{-1}(x) = \frac{1}{2} \left( \sqrt{4 \sigma^2 \frac{T}{N} + (x-l)^2} + x+l \right).
\end{equation}
\subsection{Down and out put option} \label{subsection:downoutput}

For a down-and-out put barrier option, the payoff is given by $\max(K - S_T, 0) \mathds{1}_{\zeta > T}$ where $\zeta :=\inf\{t>0:S_{t} \notin (b, + \infty)\}$, $0 < b (=e^{\ell}0$, $K$ is the option strike and $T$ the maturity. In Black-Scholes model, standard barrier option prices are given analytically and are provided for completeness (see, e.g, p.153 of \cite{haug2007complete}). It uses a common  set of factors:

\begin{align*}
A &= \phi S_0 N(\phi x_1) - \phi K N(\phi x_1 - \phi \sigma \sqrt{T}), \quad  B = \phi S_0 N(\phi x_2) - \phi K N(\phi x_2 - \phi \sigma \sqrt{T})\\
C &= \phi S_0 ( H/S )^{2(\mu + 1)} N(\eta y_1) - \phi K ( H/S )^{2 \mu} N(\eta y_1 - \eta \sigma \sqrt{T})\\
D &= \phi S_0 ( B/S )^{2(\mu + 1)} N(\eta y_2) - \phi K ( B/S )^{2 \mu} N(\eta y_2 - \eta \sigma \sqrt{T}) 
\end{align*} 
where $N$ is the cumulative distribution function of a standard Normal, 
\begin{align*} 
x_1 &= \frac{ \ln(S_0/K) }{\sigma \sqrt{T}} + (1+\mu) \sigma \sqrt{T}, \hspace{1cm}
x_2 = \frac{ \ln(S_0/H) }{\sigma \sqrt{T}} + (1+\mu) \sigma \sqrt{T}  \\
y_1 &= \frac{ \ln(H^2 / (S_0 K)) }{\sigma \sqrt{T}} + (1+\mu) \sigma \sqrt{T}, \hspace{1cm}
y_2 = \frac{ \ln(H/S) }{\sigma \sqrt{T}} + (1+\mu) \sigma \sqrt{T},
\end{align*}
$\mu = -\frac{1}{2}$ and $H = \{ b, B \}$.

For a down-and-out put barrier option with $S_0 > H = b$ and $K > H$ the price is given analytically by:
\begin{equation} \label{BSDownOutPuPrice}\nn
price = A-B+C-D+F \hspace{1cm} \eta = 1, \phi = -1.
\end{equation}

As mentioned at the beginning of this section, to put our methodology in perspective  
we have also implemented two other approaches to the numerical pricing of the barrier option:\\

\begin{itemize}

\item Standard Euler without hitting probability:\\

It consists of discretizing the SDE (\ref{BSmodelinLog}) according to the Euler scheme

\begin{equation} \label{EulerScheme}\nn
\left\lbrace
\begin{array}{l}
	\what{X}_{0} = \ln(S_0) \\
	\what{X}_{t_{i+1}} =  \what{X}_{t_{i}} - \frac{1}{2} \sigma^2 \frac{T}{N} + \sigma(W_{t_{i+1}}-W_{t_i}).
	\end{array}
\right.
\end{equation}
and evaluating $ \tilde{g}(X_T) \mathds{1}_{\zeta > T} $ by $ \tilde{g}(\what{X}_{t_{N}}) \mathds{1}_{\zeta^N > T} $
where $\zeta^N = \inf (t_i > 0: \what{X}_{t_{i}} \notin (\ell = \log(b), \infty)))$.

This numerical scheme for barrier option pricing had been studied in \cite{GobetKilled}, where it was shown to have a convergence rate of $\mathcal{O}(\frac{1}{\sqrt{N}})$. This loss of accuracy is mainly due to the fact that  it is possible for $X$ to cross the barriers $l$ or $r$ at some time $t$ between grid points $t_i$ and $t_{i+1}$ and never be below the barrier at any of the dates $t_i$ for $i=1,..,N$.\\  

\item Standard Euler with hitting probability:\\

Although this is still based on the Euler scheme simulations (\ref{EulerScheme}), it applies a further correction to  remove  the barrier crossing biases via the conditional no-hitting probability $\hat{p}_i$ using the Brownian bridge technique (see e.g, p.169 of \cite{GobetKilled}).
More precisely, the $\hat{p}_i$ are defined and can be computed analytically as  
\begin{equation} \label{nohitproba_downout}\nn
\hat{p}_i := \bbP( \forall t \in [t_i, t_{i+1}], \what{X}_{t} > l | \what{X}_{t_i} = x_i, \what{X}_{t_{i+1}} = x_{i+1})
=1-e^{\left( -2 \frac{(x_i - l)(x_{i+1} - l)}{\sigma^2 (t_{i+1}-t_i)}  \right)}
\end{equation}
where the process $(\what{X}_{t})_{0 \leq t \leq T}$ is the continuous Euler scheme which interpolates 
$(\what{X}_{t_i})_{0 \leq i \leq N }$ in the following way:

\begin{equation} \label{ContEulerScheme}\nn
\forall t \in [t_i, t_{i+1}[: \hspace{0.5cm} \what{X}_{t} = \what{X}_{t_{i}} - \frac{1}{2} \sigma^2 (t - t_i) + \sigma(W_t-W_{t_i}).
\end{equation} 
It then {\em corrects} the payoff $\tilde{g}(X_T) \mathds{1}_{\zeta > T} $ by considering instead
\begin{equation} \label{payoffwithhittingproba}\nn
 \tilde{g}(\what{X}_{t_{N}}) \prod_{i=0}^{N-1} \hat{p}_i
\end{equation} 
As shown in \cite{gobet2001euler}, this bias correction brings  the convergence rate back to of order $N^{-1}$, which is the rate of weak convergence for the Euler-Maruyama scheme in the absence of killing.  Moreover, in this specific Black-Scholes implementation, the simulation is exact, i.e no discretisation error occurs due to constant $\sigma$.
\end{itemize}

We shall next summarise the experiments details and comparison results.
\subsubsection{Set of parameters} \label{sec:dop: Set of parameters} 

The numerical experiments are conducted using the following values for the parameters:
$S_0 = 1$, $T = 1$ year, $l = \log(b = 0.8)$, $r = +\infty$ and $\sigma = 20 \%$. For thoroughness, we have considered in-the-money ($K = 1.2$), at-the-money ($K = 1$) and  out-the-money ($K = 0.9$) options. To reduce statistical noise, the simulations are run with 1 million Monte Carlo paths. The benchmark price is calculated analytically with formula (\ref{BSDownOutPuPrice}).

As our final results do not show any significant dependency on the moneyness of the option, we shall only report the results for at-the-money (ATM) options. In particular the discrepancy between benchmark prices and the numerical value for   ATM down-and-out put options is shown  in Figure  \ref{Fig_BS_ATMDownOutPutError}. We have not observed any stability issues with any of our $h$-transformation schemes. As discussed earlier, the standard Euler with hitting probability method has no discretisation error. The discrepancy is therefore essentially the statistical noise.

 Our numerical results show the rapid  convergence of the numerical approximation of prices given by the recurrent and transient transforms via the implicit scheme and demonstrate clearly its effectiveness over the standard Euler scheme without hitting probability correction. This confirms the findings of our theoretical analysis even without satisfying all the conditions of Theorem \ref{t:main}. 
 
 Moreover, the prices given by  the recurrent and transient transforms are quite comparable as predicted by the theoretical analysis. 
Figures \ref{Fig_BS_ATMDownOutPuthRecurrentLogError} and \ref{Fig_BS_ATMDownOutPuthTansientLogError} show the log-log plot of the discrepancy associated to the recurrent and transient transforms respectively for ATM down-and-out put option, respectively. The respective numerical rates of convergence observed are  $0.95$ and $0.9$.

\begin{figure}[htbp]
  \centering
  \includegraphics[scale=0.6,trim={0cm 0.8cm 0cm 0cm}, clip]{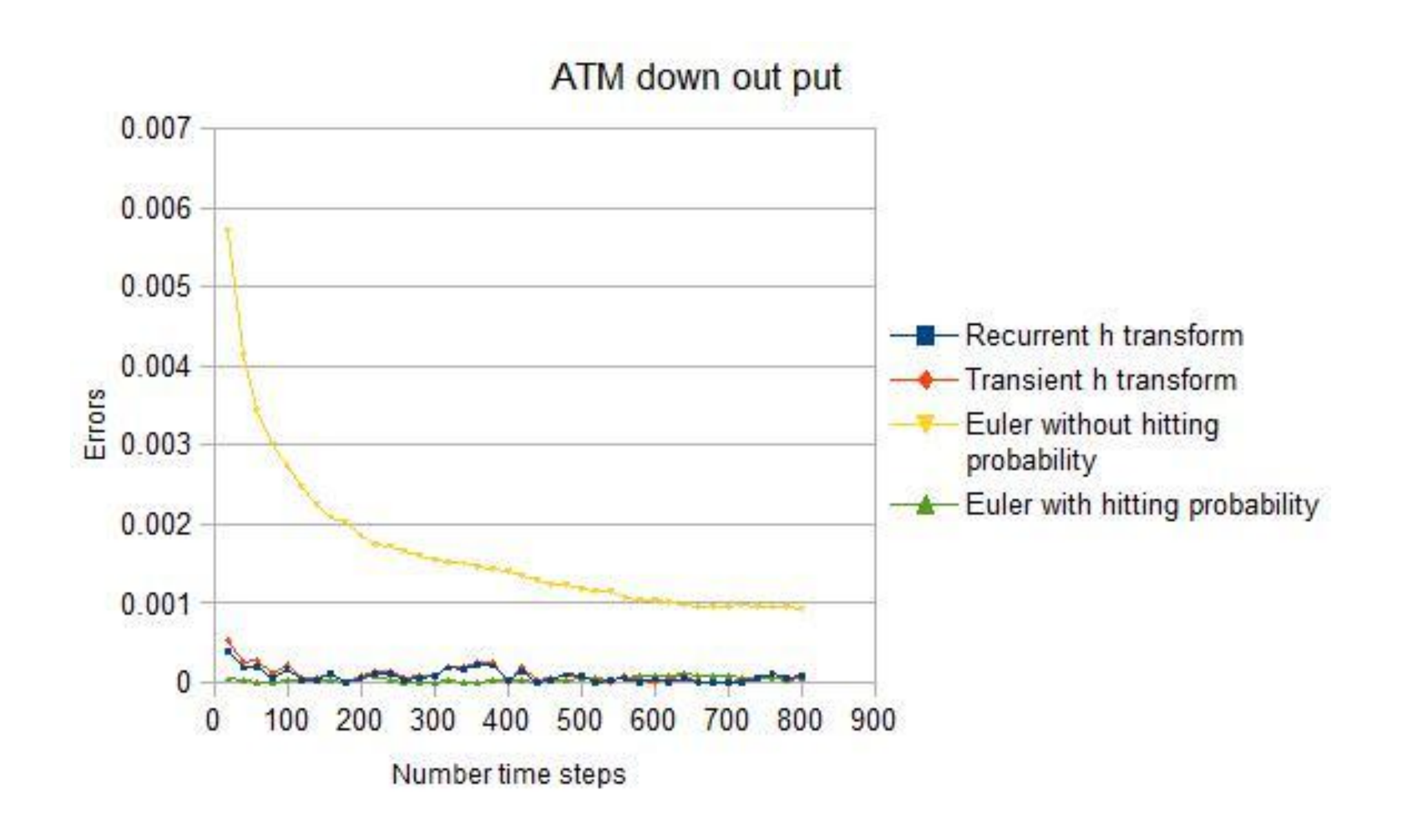}
  \caption{Absolute discrepancy between the benchmark price for ATM down-and-out put and those 
calculated with different numerical schemes when $S_0 = 1$, $K = 1$, $T = 1$ year, $l = \log(b=0.8)$, 
$r = + \infty $ and $\sigma = 20 \%$.}  
  \label{Fig_BS_ATMDownOutPutError}
\end{figure}

\begin{figure}[htbp]
  \centering
  \includegraphics[scale=0.6,trim={0cm 0.7cm 0cm 0cm}, clip]{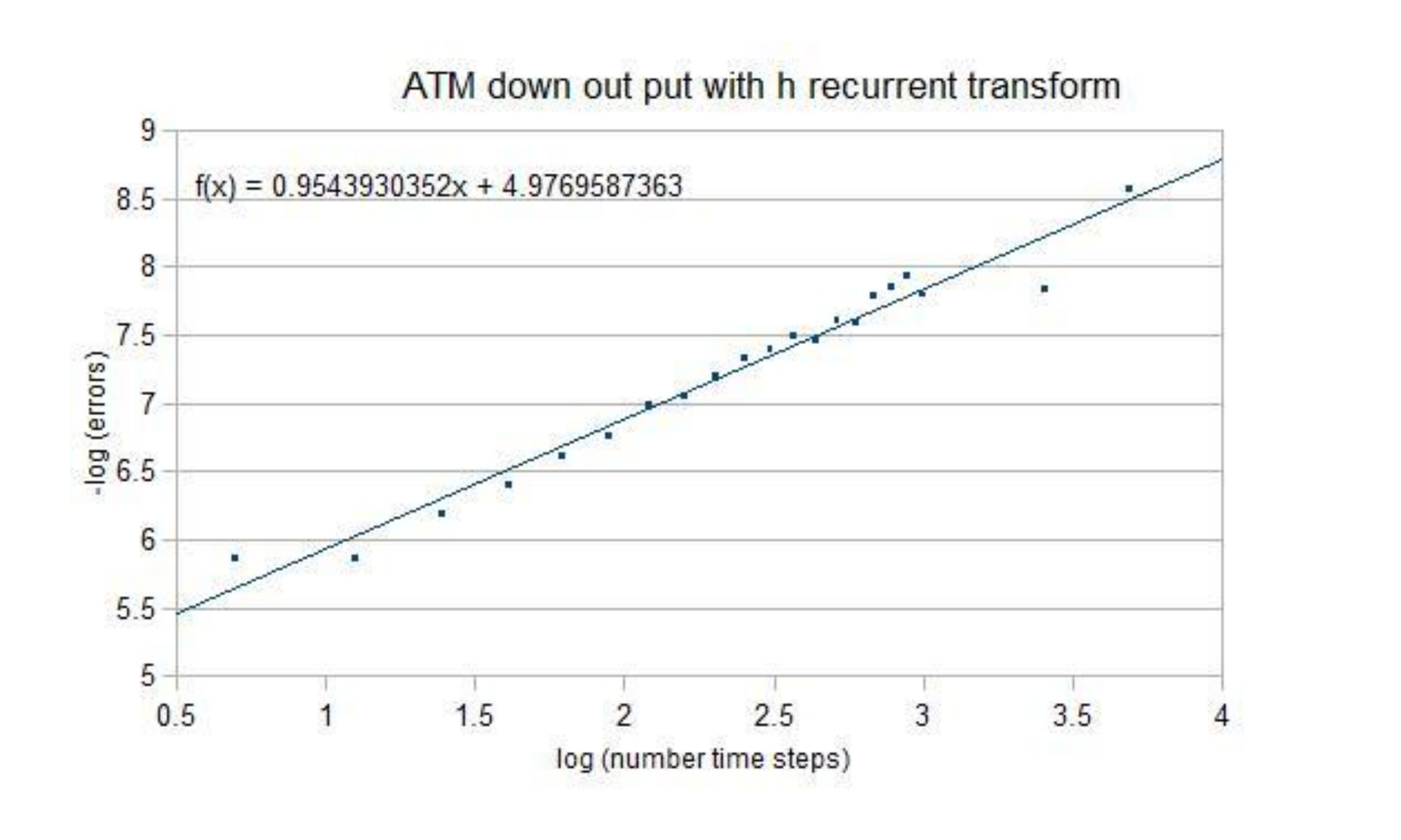}
  \caption{Log-log plot of the absolute discrepancy for ATM down-and-out put price with recurrent transform numerical scheme when $S_0 = 1$, $K = 1$, $T = 1$ year, $l = \log(b=0.8)$, 
$r = + \infty $ and $\sigma = 20 \%$.}  
  \label{Fig_BS_ATMDownOutPuthRecurrentLogError}
\end{figure}

\begin{figure}[htbp]
  \centering
  \includegraphics[scale=0.6,trim={0cm 0.7cm 0cm 0cm}, clip]{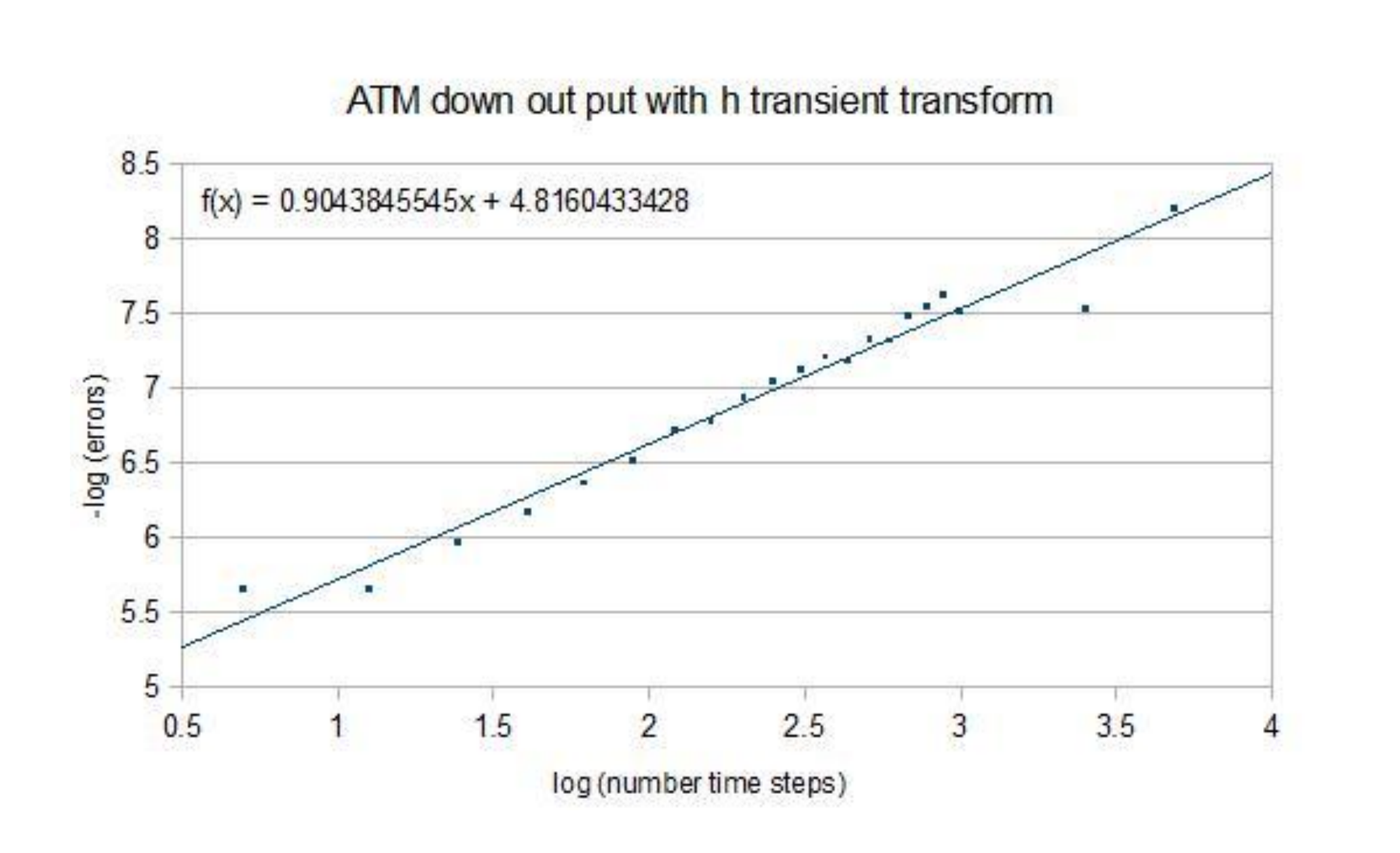}
  \caption{Log-log plot of the absolute discrepancy for ATM down-and-out put price with transient transform numerical scheme when $S_0 = 1$, $K = 1.0$, $T = 1$ year, $l = \log(b=0.8)$, 
$r = + \infty $ and $\sigma = 20 \%$.}  
  \label{Fig_BS_ATMDownOutPuthTansientLogError}
\end{figure}

\subsection{Down and up out double barrier call option} \label{subsection:Down and up out double barrier call option}

For a down-and-up-out barrier call option, the payoff is given by $\max(S_T - K, 0) \mathds{1}_{\zeta > T}$ where $\zeta :=\inf\{t>0:S_{t} \notin (b, B)$, $0 < b (=e^{\ell}) < B (=e^r)< \infty$, $K$ is the option strike and $T$ the maturity. In Black-Scholes model,  with $ b < S_0 < B $, the price can be computed using Ikeda and Kunitomo formula (see Theorem 3.2 in \cite{kunitomo1992pricing}):
\begin{align}\label{IkedakunitomoFormula}
price &= S_0 \sum_{n = - \infty}^{+ \infty} \Bigg \{ \left( \frac{B^n}{b^n} \right) [N(d_1) - N(d_2)]
- \left(  \frac{b^{n+1}}{B^n S_0} \right) [N(d_3) - N(d_4)] \Bigg \}  \\
	& - K \sum_{n = - \infty}^{+ \infty} \Bigg \{ \left( \frac{B^n}{b^n} \right)^{-1} [N(d_1 - \sigma \sqrt{T}) - N(d_2 - \sigma \sqrt{T})]
- \left(  \frac{b^{n+1}}{B^n S_0} \right)^{-1} [N(d_3 - \sigma \sqrt{T}) - N(d_4 - \sigma \sqrt{T})] \Bigg \}	\nonumber
\end{align}
where
\begin{align*}
d_{1n} &=  \frac{ \ln( S_0 B^{2n}/(K b^{2n}) ) + \sigma^2 T/2  }{\sigma \sqrt{T}}, \quad d_{2n} =  \frac{ \ln( S_0 B^{2n-1}/(b^{2n}) ) + \sigma^2 T/2  }{\sigma \sqrt{T}}\\
d_{3n} &=  \frac{ \ln( b^{2n+2}/(K S_0 B^{2n}) ) + \sigma^2 T/2  }{\sigma \sqrt{T}}, \quad d_{4n} =  \frac{ \ln( b^{2n+2}/(S_0 B^{2n+1}) ) + \sigma^2 T/2  }{\sigma \sqrt{T}}
\end{align*}

Note that the option price is expressed as an infinite series invoving weighted normal distribution functions. However, numerical studies in \cite{kunitomo1992pricing} show the convergence of the formula is rapid and it is suggested that it suffices to calculate the leading two or three terms for most cases. Here, we use the Excel spreadsheet provided in \cite{haug2007complete} which computes each series above with $n$ from $-5$ to $5$.    

For the standard Euler with hitting probability correction, the no-hitting probability $\hat{p}_i$ is also given as an infinite series in \cite{GobetKilled}\footnote{up to  a typographical error.} 

\begin{align} 
\hat{p}_i &:= \bbP( \forall t \in [t_i, t_{i+1}], \what{X}_{t} \in (\ell,r)  | \what{X}_{t_i} = x_i, \what{X}_{t_{i+1}} = x_{i+1})\nn \\
&=  \mathds{1}_{ l < x_i, x_{i+1}  <r} \sum_{n=-\infty}^{n=+\infty}  \left[ e^{ \frac{-2n (r-l)(n(r-l)+ x_{i+1} - x_{i})}{\sigma^2 (t_{i+1} - t_{i})} } - e^{ \frac{-2( n (r-l) + x_{i} -r )( n(r-l) + x_{i+1} - r )}{\sigma^2 (t_{i+1} - t_{i})} } \right] \label{nohitproba_doublenotouch}
\end{align}

The practical studies on this formula  again suggest it is perfectly sufficient for numerical purposes to calculate the leading two or three terms in most cases. To be conservative, in our experiments, the $\hat{p}_i$ are estimated using $n$ from $-5$ to $5$.
\subsubsection{Set of parameters} \label{subsubsection:Down and up out double barrier call option Set of parameters} 
The numerical experiments are conducted using the following values for the parameters:
$S_0 = 1$, $T = 1$ year, $b = 0.85$, $B = 1.25$ and $\sigma = 20 \%$. For thoroughness, we consider in-the-money ($K = 0.9$), at-the-money ($K = 1$) and  out-the-money ($K = 1.05$) options. To reduce statistical noise, the simulations are run with 1 million Monte Carlo paths. The benchmark price is calculated with formula (\ref{IkedakunitomoFormula}) with truncation by keeping terms from  $n = -5$ to $n = 5$.

As in the previous study, no significant difference is observed by changing the moneyness of the option. Thus, we we will again report the results pertaining to the ATM options. The discrepancies between benchmark prices and numerical methods for  the ATM down-and-up-out call options are shown in Figure  \ref{Fig_BS_ATMDoubleBarrierCallError}. We have not observed any stability issues with the  recurrent transform method. As discussed, the standard Euler with hitting probability correction has only truncation error in the computation of no hitting probability (\ref{nohitproba_doublenotouch}), which  we believe to be negligible. The discrepancy can  then be attributed essentially to the statistical noise. As in the previous experiment our numerical results align with the theoretical predictions.  In particular  Figure  \ref{Fig_BS_ATMDoubleBarrierCallLogError} shows the log-log plot of the discrepancy associated to the recurrent transform method for ATM  down-and-up-out call options with numerical rate of convergence of  $0.81$.

\begin{figure}[htbp]
  \centering
  \includegraphics[scale=0.6,trim={0cm 0.8cm 0cm 0cm}, clip]{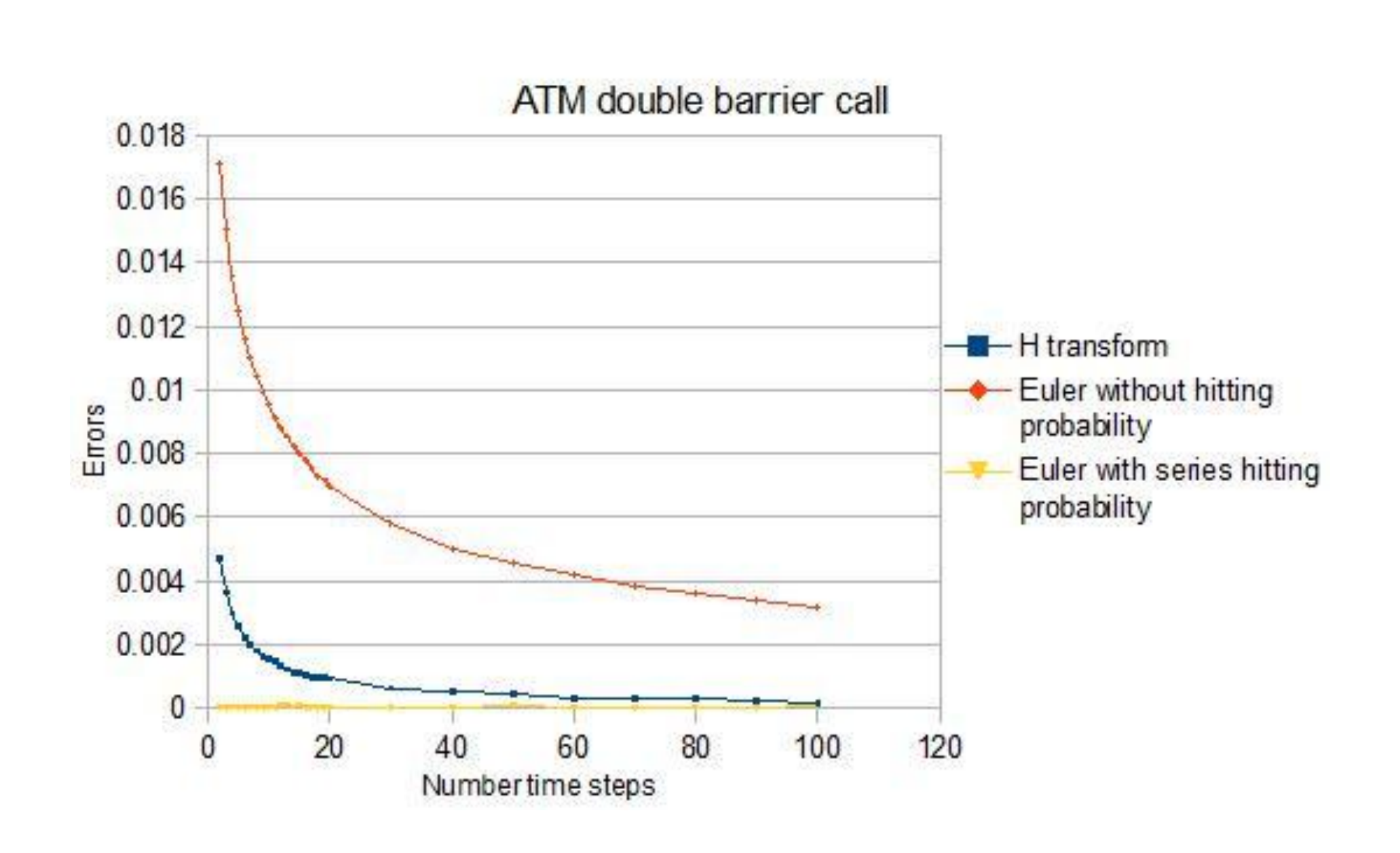}
  \caption{Absolute discrepancy between the benchmark price for ATM double barrier call and those 
calculated with different numerical schemes when $S_0 = 1$, $K = 1$, $T = 1$ year, $b = 0.85$, $B  = 1.25$ and $\sigma = 20 \%$.}  
  \label{Fig_BS_ATMDoubleBarrierCallError}
\end{figure}

\begin{figure}[htbp]
  \centering
  \includegraphics[scale=0.6,trim={0cm 0.7cm 0cm 0cm}, clip]{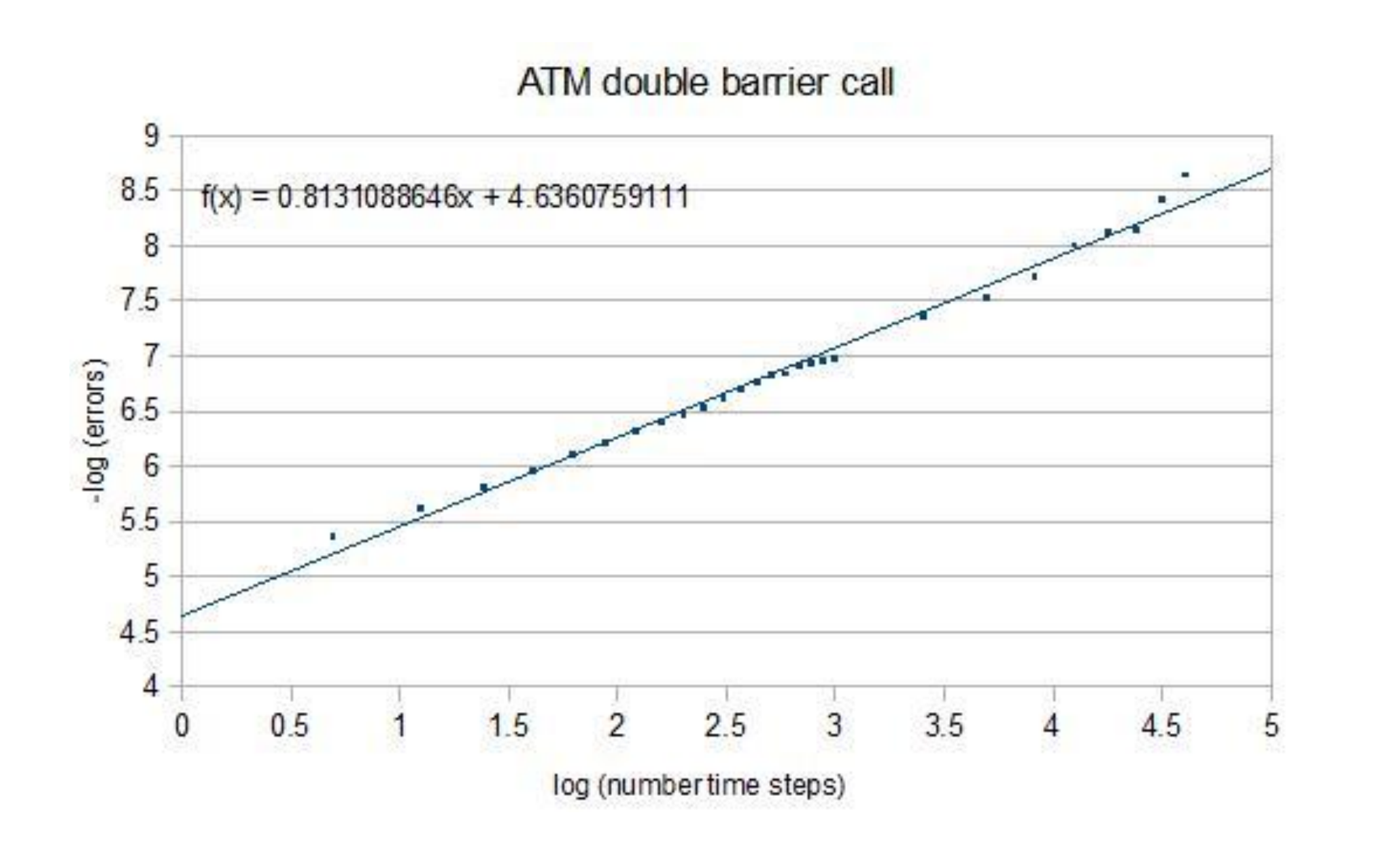}
  \caption{Log log plot of the absolute discrepancy for ATM double barrier call price with recurrent transform numerical scheme when $S_0 = 1$, $K = 1$, $T = 1$ year, $b = 0.85$, $B  =  1.25$ and $\sigma = 20 \%$.}  
  \label{Fig_BS_ATMDoubleBarrierCallLogError}
\end{figure}

\subsection{Time-homogeneous hyperbolic local volatility model} \label{subsec:HLV model}

Since the advent of the Black-Scholes option pricing formula, the study of implied volatility has become a central preoccupation for both academics and practitioners. It is well known, actual option prices rarely conform to the predictions of explicit formulas because the idealized assumptions required for it to hold don't apply in the real world. Consequently, implied volatility (the volatility input to the Black-Scholes formula that generates the market European Call or Put price) in general depends on the strike $K$ and the maturity of the option $T$. The collection of all such implied volatilities is known as the volatility surface. For example, the effect that implied volatility $\sigma_{im}(T,K)$ is a decreasing
function of strike is called {\it{skew}} and is usually observed in equity derivatives market. This means that the underlying asset price process cannot be explained using the Black-Scholes
model, for which the implied volatility does not depend on the strike.  This motivates the researchers  to find a convenient model for the underlying asset to evaluate contingent
claim prices. Local volatility models, either parametric or non-parametric, (see
e.g \cite{Dupire94,DerKa98,Rubi94}), arguably capture the surface of implied volatilities more precisely than other approaches such as stochastic volatility models (see e.g \cite{MadQianRen07,Romo12}). Needless to say, the volatility surface has a significant impact on barrier option valuation. Indeed, the barrier hitting probability depends strongly on the dynamics of the volatility of  the spot pricess (see, e.g., \cite{Bossens19}). 

For our analysis, we consider the time homogeneous hyperbolic local volatility
model (HLV), which is widely used in quantitative finance community to capture the
market {\it{skew}}. It corresponds to a parametric local volatility-type model in
which the dynamics of the underlying under the risk neutral measure $\bbP$ is given by

\begin{equation} \label{HLVmodel}\nn
dX_t = \sigma(X_t) dW_t, \,\,\,\, X_0=1,
\end{equation}
where
\begin{equation} \label{HLVVol}\nn
\sigma (x) = \nu \Big\{ \frac{(1-\beta+\beta^2)}{\beta} x +\frac{(\beta-1)}{\beta}  \big(\sqrt{x^2+\beta^2(1-x)^2}-\beta\big) \Big\}.
\end{equation}
Here $\nu >0$ is the level of volatility, $\beta \in (0,1]$ is the skew
parameter.

First introduced in \cite{Jackel10} it behaves similarly to the Constant Elasticity of Variance (CEV) model and has been used for numerical experiments in, e.g.,  \cite{HokNGareAntonis18,HokShih2019,hok2021pricing}.
A practical  advantage of this model is that zero is not an attainable boundary, which in turn   avoids
some numerical instabilities present in the CEV model when the underlying asset price
is close to zero (see e.g. \cite{And00}). It corresponds to the Black-Scholes model for $\beta=1$ and exhibits a skew for the implied volatility surface when $\beta \neq 1$. Figure \ref{figure:HyperbolicLVImpactBeta} illustrates the impact of the parameter $\beta$ on the skew of the volatility surface. We observe that the skew increases significantly with decreasing value of $\beta$. For example with $\nu = 0.3, \, \beta = 0.2$, the difference in volatility between strikes at $50\%$ and at $100\%$ is about $15 \%$.

\begin{figure}[htbp]
  \centering
 \includegraphics[scale=0.45, trim={0.5cm 7cm 1cm 6cm}, clip]{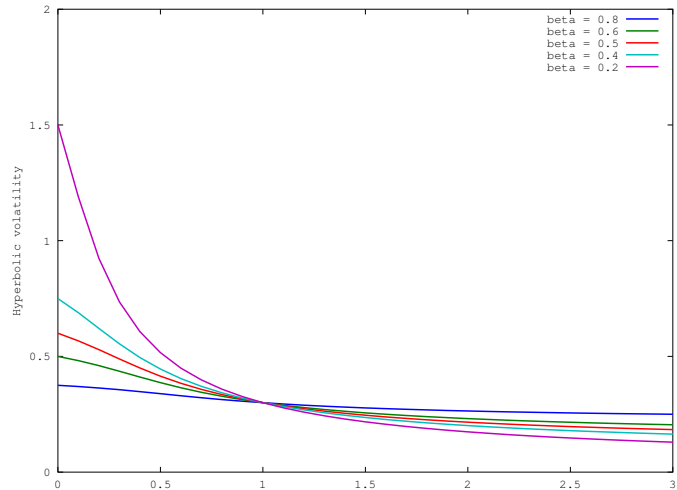}
  \caption{ Impact of the value $\beta$ on the hyperbolic local volatility for fixed volatility level $\nu = 0.3$. }
  \label{figure:HyperbolicLVImpactBeta}
\end{figure}

\subsubsection{Down and up out double barrier call option} \label{subsubsection:Down and up out double barrier call option}
In this implementation we shall set $h(x) = (x-l)(r-x)$
 and the associated BEM scheme will be then solved using bisection method with Octave vectorization for faster code execution. Consequently the price  is  approximated by
\begin{equation*} \label{BEMPriceHLV}
price \approx h(x) \bbE^{h,x} \left[ \frac{(\what{X}_{t_{N}} - K)_+}{h(\what{X}_{t_{N}})} e^{ \frac{1}{2} \frac{T}{N} \sum_{n=0}^{N-1} \sigma^2(\what{X}_{t_{n}}) \frac{h''}{h}( \what{X}_{t_{n}} ) }  \right] 
\end{equation*}

For comparison, we compute also the numerical price given by the standard Euler scheme with hitting probability. The scheme is given by equation (\ref{EulerScheme}) and the no hitting probability formula by (\ref{nohitproba_doublenotouch}), where $\sigma$ is computed using the parametric local volatility function (\ref{HLVVol}).  Experiment details and comparison results are described below.

\subsubsection{Set of parameters} \label{sec:hlv: Set of parameters} 

The numerical experiments are conducted using the following values for the parameters:
$S_0 = 1$, $\nu = 20 \%$, $\beta = 0.5$, $T = 1$ year, $b = 0.85$, $B = 1.25$.   For thoroughness, we consider in-the-money ($K = 0.9$), at-the-money ($K = 1$) and  out-the-money ($K = 1.05$) options. The benchmark prices for each numerical method are computed by the method itself with very dense time grid and high number of Monte Carlo paths.

In this case we observed some differences regarding the moneyness of the option in our numerical results. More precisely, the method performed relative poorly for the ATM option. For this reason we report below the results in all three cases and provide an explanation for the seemingly poor performance for the ATM option.

The discrepancies between benchmark prices and numerical methods for ITM, ATM and OTM double barrier call options are shown respectively in Figures \ref{Fig_HLV_ITMUpDownOutCallError}, \ref{Fig_HLV_ATMUpDownOutCallError} and \ref{Fig_HLV_OTMUpDownOutCallError}. We have not observed any stability issues with the recurrent transform scheme.  Interestingly, our recurrent transformation has a much smaller error than  the explicit Euler method with hitting probability correction when the number of discretisations is reasonably large. More importantly, this outperformance is still valid even if the number of Monte Carlo simulations for the explicit Euler method is increased five times. Having said that, one should still treat such a conclusion with caution as our benchmark price and hitting probabilities are calculated by applying a truncation and, thus, is subject to error. Nevertheless, the {\em outperformance} is still promising as our truncation is no coarser than the common industry practice.

 Figures \ref{Fig_HLV_ITMDoubleBarrierCallLogError}, \ref{Fig_HLV_ATMDoubleBarrierCallLogError} and \ref{Fig_HLV_OTMDoubleBarrierCallLogError} show, respectively, the log-log plot of the discrepancy associated to the recurrent transform method for ITM, ATM and OTM double barrier call options. The numerical rate of convergence are respectively $0.91$, $0.63$ and $1$, using $2\times 10^5$ Monte Carlo simulations. Although the rate of convergence for the ATM option is far from the theoretical rate of 1, a closer look at Figure \ref{Fig_HLV_ATMUpDownOutCallError} reveals a clue. Note that the error of approximation converges very rapidly to zero after a few iterations and further discretisations do not significanly alter the already very small error term. This indicates that the observed error in this case can be mostly attributed to the statistical noise and the simple regression to obtain the convergence rate does not work well. 
 
 When we run the same experiment for the Euler scheme with hitting probability correction with $2\times 10^5$ Monte Carlo simulations, we observe a similar drop in the performance and the convergence rates are found to be  $0.50$, $0.59$ and $0.61$, respectively.  However, the convergence rates for the latter scheme increases to  $0.83$, $0.83$ and $0.77$, respectively, when the number of simulations are increased five-fold. 

\begin{figure}[htbp]
  \centering
  \includegraphics[scale=0.6,trim={0cm 1.2cm 0cm 0cm}, clip]{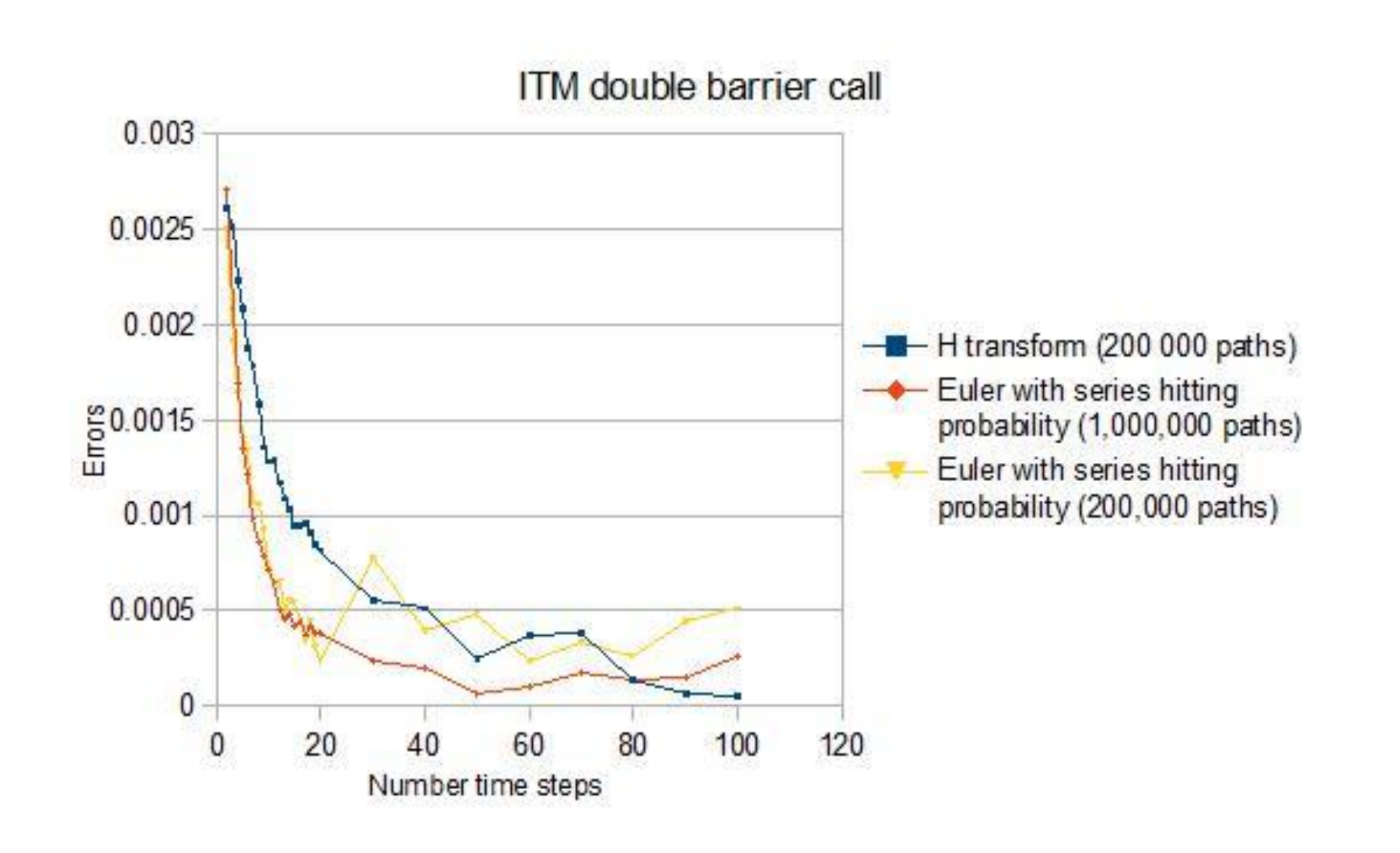}
  \caption{Absolute discrepancy between the benchmark price and those 
calculated by different numerical schemes for ITM double barrier call when $S_0 = 1$, $K = 0.9$, $\nu = 20 \%$, $\beta = 0.5$, $T = 1$ year, $b = 0.85$, $B = 1.25$.}  
  \label{Fig_HLV_ITMUpDownOutCallError}
\end{figure}

\begin{figure}[htbp]
  \centering
  \includegraphics[scale=0.6,trim={0cm 0.8cm 0cm 0cm}, clip]{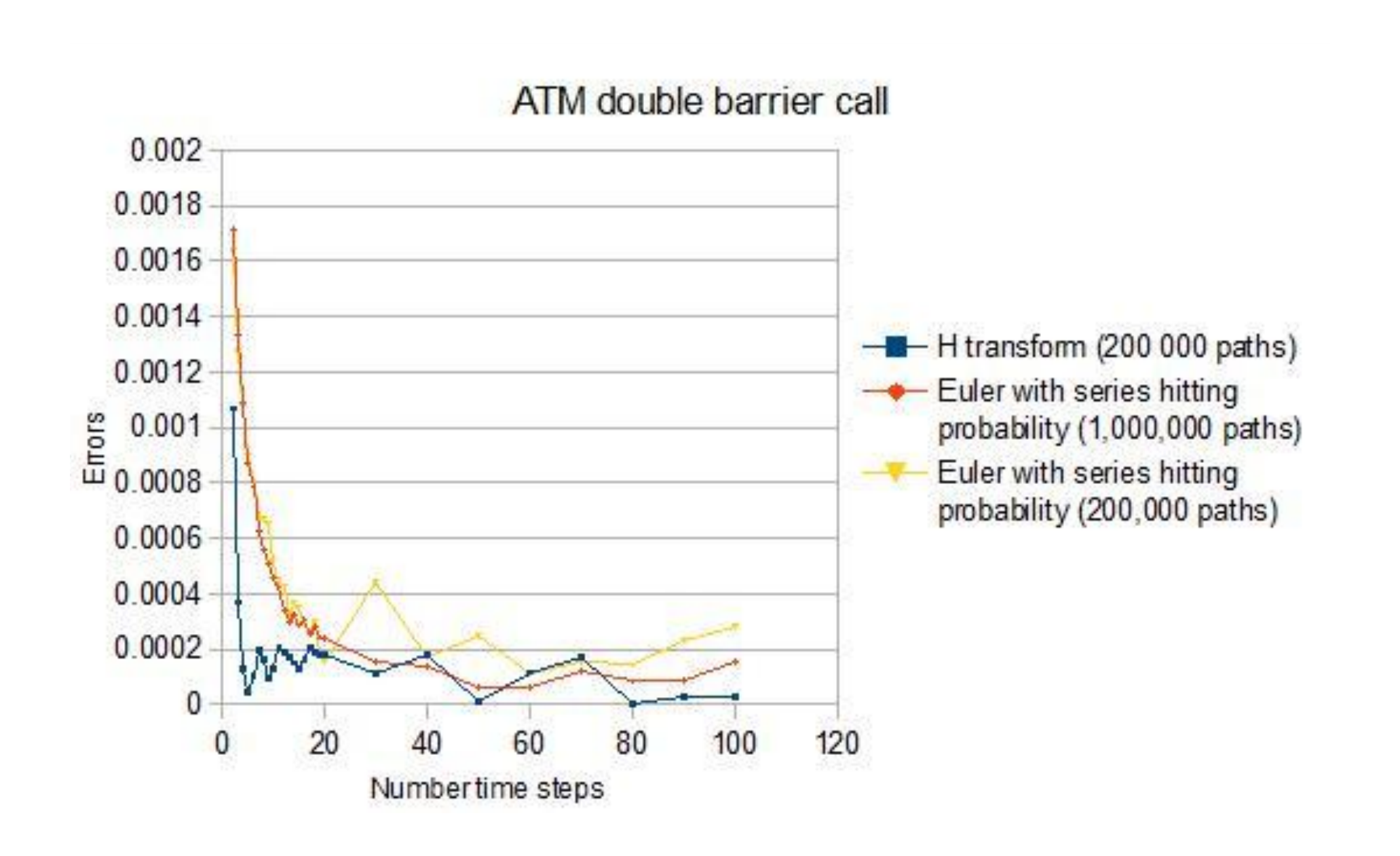}
  \caption{Absolute discrepancy between the benchmark price for and those 
calculated by different numerical schemes for ATM double barrier call when $S_0 = 1$, $K = 1$, $\nu = 20 \%$, $\beta = 0.5$, $T = 1$ year, $b = 0.85$, $B = 1.25$.}  
  \label{Fig_HLV_ATMUpDownOutCallError}
\end{figure}

\begin{figure}[htbp]
  \centering
  \includegraphics[scale=0.6,trim={0cm 0.8cm 0cm 0cm}, clip]{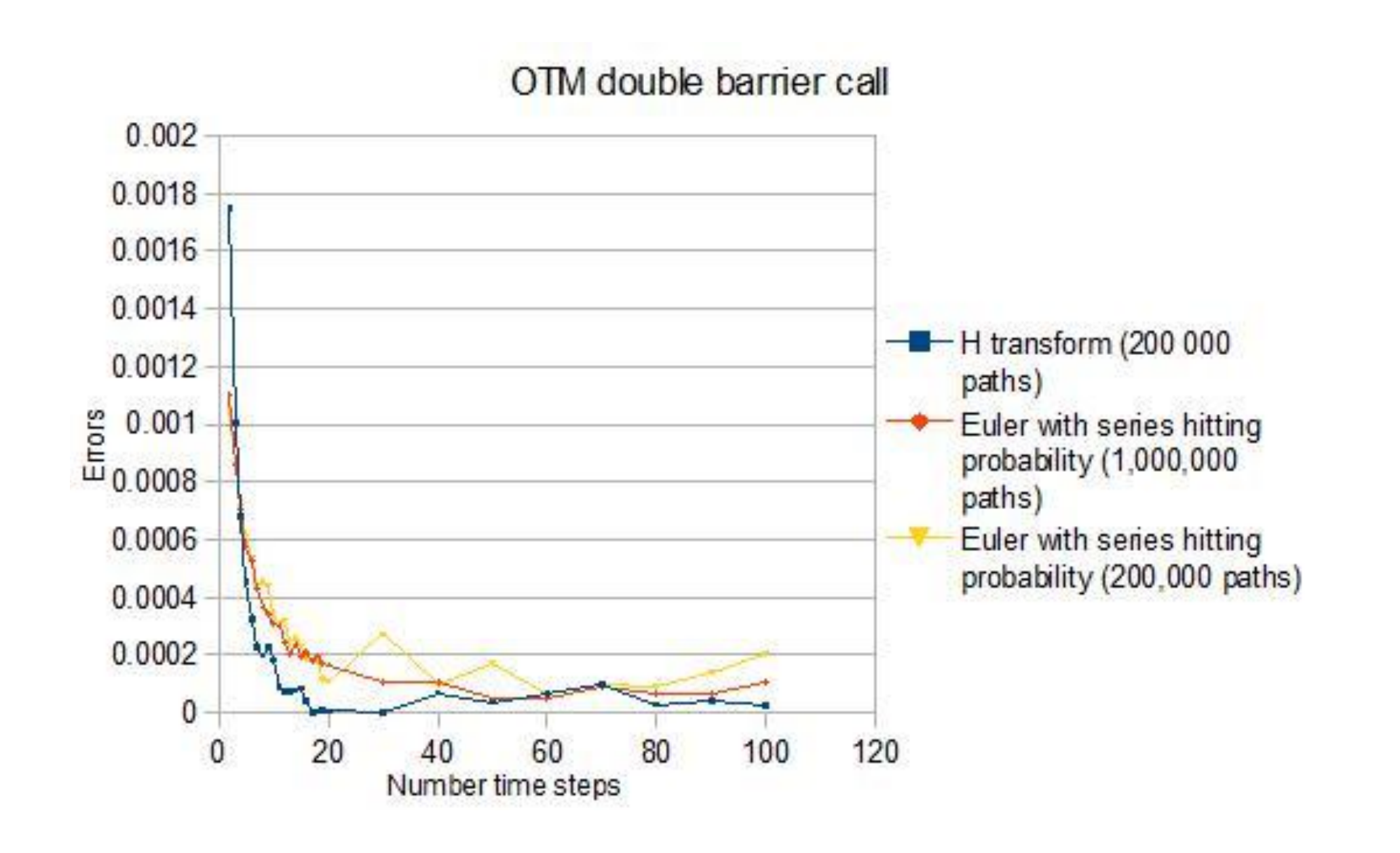}
  \caption{Absolute discrepancy between the benchmark price and those 
calculated by different numerical schemes for double barrier call when $S_0 = 1$, $K = 1.05$, $\nu = 20 \%$, $\beta = 0.5$, $T = 1$ year, $b = 0.85$, $B = 1.25$.}  
  \label{Fig_HLV_OTMUpDownOutCallError}
\end{figure}

\begin{figure}[htbp]
  \centering
  \includegraphics[scale=0.6,trim={0cm 0.7cm 0cm 0cm}, clip]{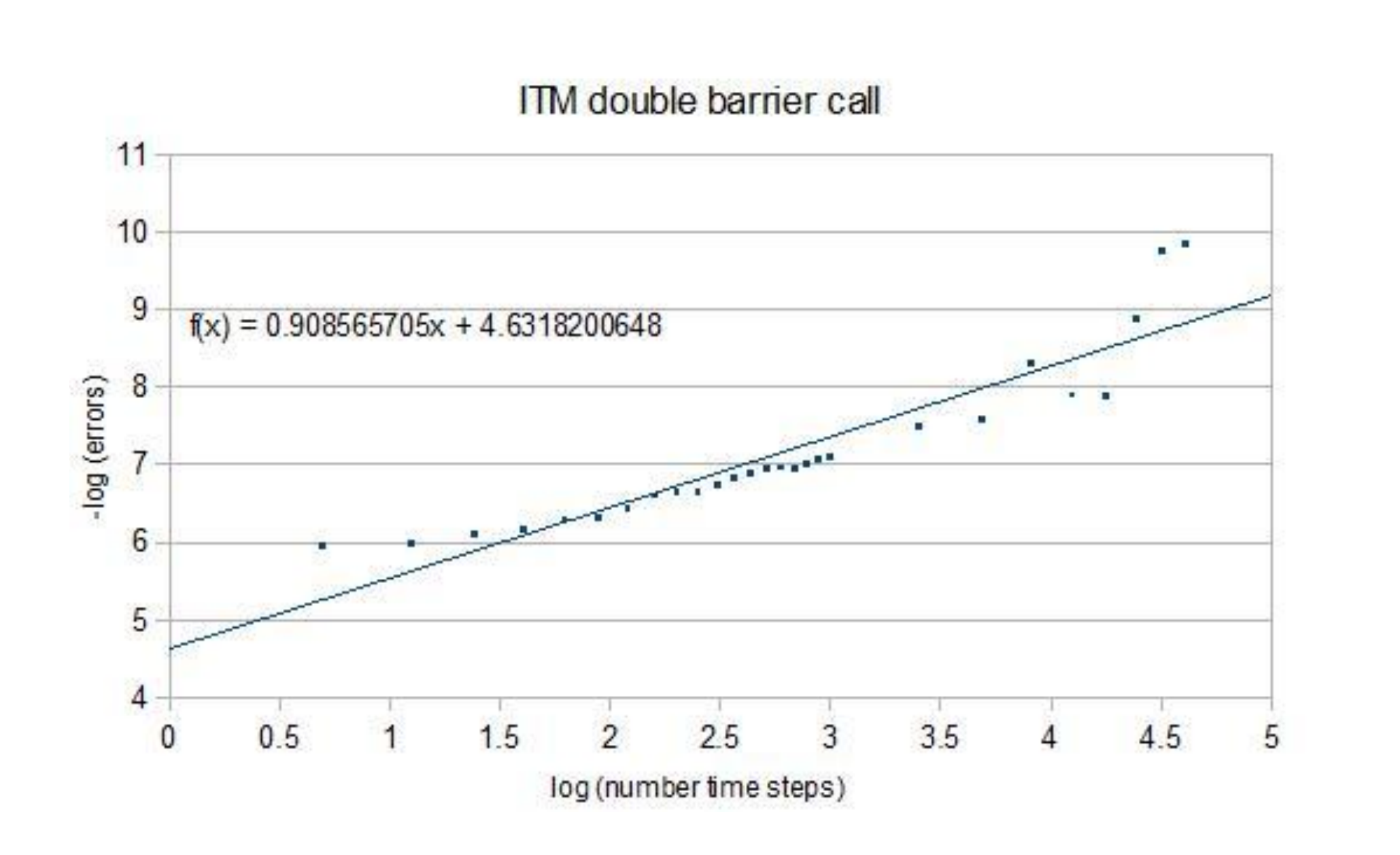}
  \caption{Log-log plot of the absolute discrepancy for ITM double barrier call price with H-transform numerical scheme when $S_0 = 1$, $K = 0.9$, $T = 1$ year, $b = 0.85$, $B = 1.25$, $\nu = 20 \%$ and $\beta = 0.5$.}  
  \label{Fig_HLV_ITMDoubleBarrierCallLogError}
\end{figure}

\begin{figure}[htbp]
  \centering
  \includegraphics[scale=0.6,trim={0cm 0.7cm 0cm 0cm}, clip]{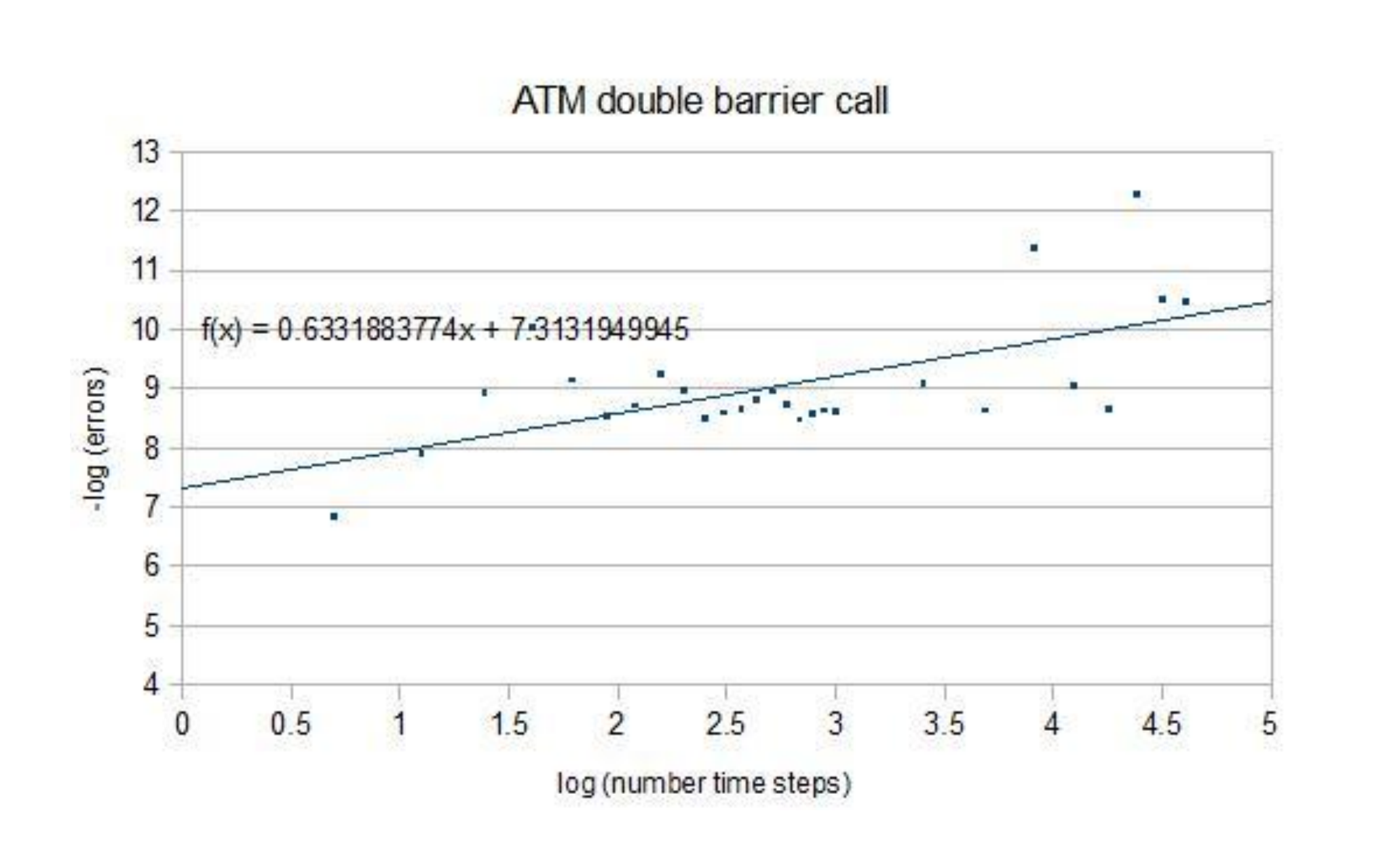}
  \caption{Log-log plot of the absolute discrepancy for ATM double barrier call price with H-transform numerical scheme when $S_0 = 1$, $K = 1$, $T = 1$ year, $b = 0.85$, $B  =  1.25$, $\nu = 20 \%$ and $\beta = 0.5$.}  
  \label{Fig_HLV_ATMDoubleBarrierCallLogError}
\end{figure}

\begin{figure}[htbp]
  \centering
  \includegraphics[scale=0.6,trim={0cm 0.7cm 0cm 0cm}, clip]{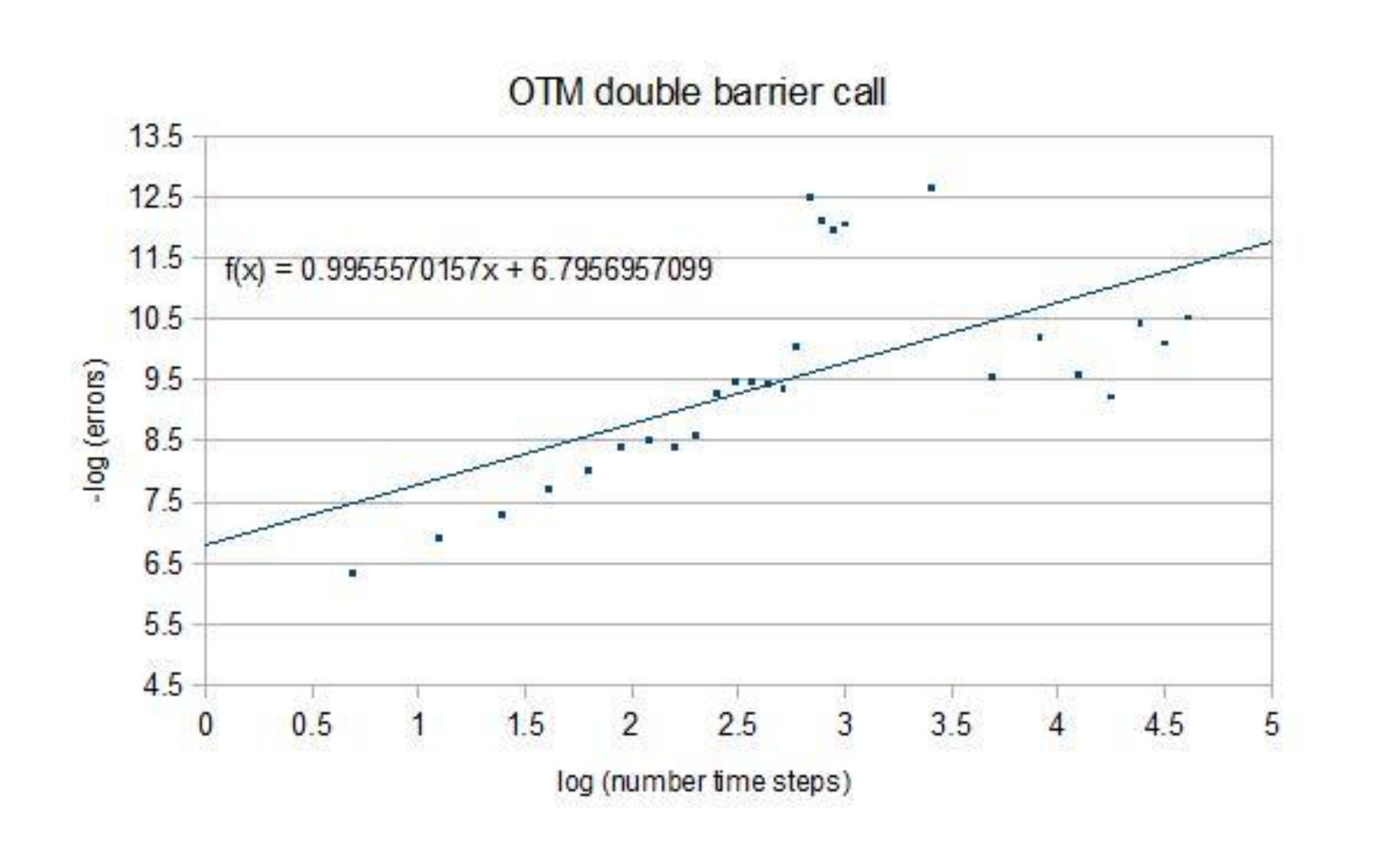}
  \caption{Log-log plot of the absolute discrepancy for OTM  Double Barrier Call price with H-transform numerical scheme when $S_0 = 1$, $K = 1.05$, $T = 1$ year, $b = 0.85$, $B=1.25$, $\nu = 20 \%$ and $\beta = 0.5$.}  
  \label{Fig_HLV_OTMDoubleBarrierCallLogError}
\end{figure}
\section{Conclusion} \label{s:conclusion}
 We have introduced a novel backward Euler-Maruyama method to increase the weak convergence rate of approximations in the presence of killing. The numerical experiments confirm our theoretical prediction that the convergence rate is of order $1/N$, where $N$ is the number of discretisations. Moreover, the numerical studies suggest that one does not need a large $N$ to obtain a sufficiently close approximations as all numerical studies indicate errors terms diminishing very rapidly with a small number of iterations. The numerical experiments also suggested our method  outperforming the {\em Brownian bridge method} in certain cases although such a statement does not currently have any theoretical backing. However, we believe that the method developed in this paper will perform better when applied to a higher order Euler-scheme such as  the Milstein scheme. Such investigations are left for future research.
 
 Moreover, a close look into our technical analysis reveals that our convergence result does not depend heavily on the one-dimensional nature of the problem. In particular it is relatively clear how to obtain a version of Theorem \ref{t:invmomnt} in the multidimensional case using  well-established potential theoretic arguments. However, our main obstacle in not being able to immediately obtain a multidimensional version of Theorem \ref{t:main} is the absence of a systematic study of recurrent transformations in higher dimensions. Such a study and its applications to the Euler methods for killed diffusions will be the subject of future research. 
\bibliographystyle{siam}
\bibliography{ref}
\appendix
\section{Proof of Theorem \ref{t:bemest}}
Proof will be divided into several steps considering first the case of $r=\infty$  and making use of the comparison Lemma \ref{l:tchange}. In what follows $K$ denotes a generic constant independent of $N$.

\begin{enumerate}
	\item First suppose $r=\infty$. Since $1/h$ is decreasing,  Lemma \ref{l:tchange} and Theorem \ref{t:invmomnt} imply
	\[
	\sup_{t\leq T, N} E^{h,X_0}\left(\frac{1}{h}(\wXt)\right)<\infty.
	\]
	Moreover, Lemma \ref{l:tchange} also yields
	\[
	E^{h,X_0}\sum_{n=0}^{N-1}\int_{t_n}^{t_{n+1}}\frac{\sigma^2(\wXn)h^{-2-p}(\wXt)}{H_x^{2}(t_n,\wXn;t,\wXt)}dt\leq E^{h,X_0}\int_0^{A_T}\frac{1}{h^{2+p}(Y_t)}dt\leq E^{h,X_0} \int_0^{\|\sigma\|_{\infty}^2T}\frac{1}{h^{2+p}(Y_t)}dt<\infty,
	\]
	where $Y$ is a process that shares the same law with the process in Theorem \ref{t:invmomnt} with $c=c_1$ and the last inequality follows from Theorem \ref{t:invmomnt}. 
	
	Similarly, by considering instead the time change 
	\[
	dA_t=\frac{\sigma^2(\wXn)}{H_x^2(t_n,\wXn;t,\wXt)}dt,\qquad t\in (t_n,t_{n+1}), \qquad A_{t_n}=t_n,
	\]
	we obtain $h^{-p}(\what{X}_{\tau})\leq h^{-p}(Y_{A_{\tau}})$, where $Y$ is a process such that $Y_{t_n}=\wXn$ and
	\[
	dY_t=d\beta_t+ \left(\frac{h'}{h}(Y_t)+c_1\right)dt, \quad t\geq t_n,
	\]
	with $\beta$  being a standard Brownian motion. Consequently, Theorem \ref{t:invmomnt} yields
	\[
	\esssup_{\tau \in \cT_n}E^{h,X_0}\Big(\frac{1}{h}(\what{X}_{\tau})\big| \cF_{t_n}\Big)<\infty
	\]
	since $A_{\tau}\leq t_{n}+\|\sigma\|_{\infty}^2\frac{T}{N}$, a.s. for $\tau\in \cT_n$.
	\item Now consider the case $r<\infty$ and set $x_1:=\inf\{x\geq 0:h'(x)=0\}$ and $x_2:=\inf\{x\geq x_1 :h'(x)<0\}$. Then, there exist functions $h_1$ and $h_2$ such that $h=h_1 h_2$, $h_1$ (resp. $h_2$) is non-decreasing (resp. non-increasing) and constant on $(x_1, r)$ (resp. $(0,x_2)$).
	
	Let's define the processes $\what{Y}^i$, where $\what{Y}^i_0=X_0$ and
	\[
	d\what{Y}^i_t=\frac{\sigma(\wXn)}{H_x(t_n,\wXn;t,\wYt^i)}dW_t +\frac{\sigma^2(\wXn)}{H_x^2(t_n,\wXn;t,\wYt^i)}\left(\frac{h'_i}{h_i}(\wYt^i)+c_i\right) dt, \quad t\in (t_n,t_{n+1}].
	\]

	Applying Ito formula to $((x_2-\wYt^1)^+)^2$ and $((x_2-\wXt)^+)^2$, the comparison theorem employed in Lemma \ref{l:tchange}  shows that
	\[
	P^{h,X_0}(\wYt^1\wedge x_2\leq \wXt\wedge x_2, t\leq T)=1.
	\]
	An analogous argument also shows that
	\[
	P^{h,X_0}(\wYt^2\vee x_1\geq \wXt\vee x_1, t\leq T)=1
	\]
	as well.
	
	As $h_1$ is non-decreasing, $h_2$ is non-increasing and $\frac{h}{h_1}$ (resp. $\frac{h}{h_2}$) is constant on $(0, x_2)$ (resp. $(x_1,r)$), the above comparisons imply that $\frac{1}{h(\wXt\wedge x_2)}\leq \frac{1}{h(\wYt^1\wedge x_2)}$ and $\frac{1}{h(\wXt\vee x_1)}\leq \frac{1}{h(\wYt^2\vee x_1)}$.
	
	Thus, the same time change argument from Lemma \ref{l:tchange} yields that
	\[
	\sup_{t\leq T, N} E^{h,X_0}\left(\frac{1}{h}(\wXt)\right)<\infty
	\]
	by another application of Theorem \ref{t:invmomnt}\footnote{Although $h_2$ does not quite satisfy the condition therein, we obtain the result that we need by a change of scale and considering instead the function $h$ defined by $h(x)=h_2((r-x)^+)$. Note that $h$ is still continuously differentiable.}. 
	This readily implies
	\[
	\sup_{t\leq T, N} E^{h,X_0}\left(\frac{1}{h(\wXt)}\right)\leq  \frac{K'}{h(X_0)},
	\]
	for some $K'$ that depends only on  $T$. 
	
	Similarly,
	\[
	\sup_N E^{h,X_0}\sum_{n=0}^{N-1}\int_{t_n}^{t_{n+1}}\frac{\sigma^2(\wXn)h^{-2-p}(\wXt)}{H_x^{2}(t_n,\wXn;t,\wXt)}dt<\infty,
	\]
	in view of Theorem \ref{t:invmomnt} again.
	
	Analogous considerations also yield
	\[
	\esssup_{\tau \in \cT_n}E^{h,X_0}\Big(\frac{1}{h}(\what{X}_{\tau})\big| \cF_{t_n}\Big)<\infty.
	\]
	\item We shall now show the boundedness of the moments.  Note that there is nothing to show when $r<\infty$.  So,  let's assume that $r=\infty$.
	Recall that
	\[
	\wXt=\wXn + \sigma^2(\wXn)(t-t_n)\frac{h'}{h}(\wXt)+ \sigma(\wXn)(W_t-W_{t_n}).
	\]
	Thus,
	\[
	E^{h,X_0}(\wXt)\leq E^{h,X_0}(\wXn)+K(t-t_n)
	\]
	for some $K$ due to the boundedness of $\sigma$ and $h'$ as well as the uniform bound on the inverse moment of $h(\wXt)$. This shows that
	\[
	\sup_{t\leq T, M}E^{h,X_0}(\wXt)\leq X_0+KT.
	\]
	Now, suppose that 
	\[
	E(m):=\sup_{t\leq T, N}E^{h,X_0}(\wXt^m)< \infty,
	\]
	and deduce from (\ref{e:Xhdecomp}) that
	\[
	\begin{split}
	d\wXt^{m+1}&=dZ_t+(m+1)\frac{\wXt^{m}\sigma^2(\wXn)}{H_x^2(t_n,\wXn;t,\wXt)}\left\{\frac{h'}{h}(\wXt) +\mu(t_n, \wXn; t, \wXt)\right\} dt\\
	&+ \half m(m+1) \frac{\wXt^{m-1}\sigma^2(\wXn)}{H_x^2(t_n,\wXn;t,\wXt)}dt,
	\end{split}
	\]
	where $Z$ is a local martingale.
	
	Next observe that for $m\geq 1$ 
	\be \label{e:momest1}
	x^m /h \leq K(1+ x^{m-1})
	\ee
	as $h(0)=0$, $h'(0)>0$ and $h'/h\leq \frac{1}{x}$. The last identity follows from the fact that
	\[
	h(x)=-\int_0^xy h''(y)dy +xh'(x).
	\]
	Moreover, the representation of $\mu$ from (\ref{e:murep}) and (\ref{e:u2u1}) show that
	\be \label{e:momest2}
	|\mu|\leq K(H_x+1)\frac{1}{h}
	\ee
	since the term in front of the parentheses in (\ref{e:murep}) is bounded. 
	
	Observe that $(\tau_k)_{k\geq 1}$, where $\tau_k:=\inf\{t\geq t_n: \wXt\geq k\}$ is a localising sequence for $Z$. Therefore, a standard localisation argument, (\ref{e:momest1}) and (\ref{e:momest2}) together imply for $t\in (t_n,t_{n+1}]$
	\[
	E^{h,X_0}(\wXt^{m+1})\leq E^{h,X_0}(\wXn^{m+1})+ (t-t_n)K E(m-1),
	\]
	in view of the Fatou's lemma for some constant $K$, which in turn yields
	\[
	E(m+1)\leq X_0^{m+1}+ KT E(m-1).
	\]
	Finally, note that this in particular implies that $Z$ is a true martingale. Thus, for $\tau\in \cT_n$ and $m\geq 2$
	\[
	\what{X}^m_{\tau}\leq \wXn^m + M_{\tau} +K \int_{t_n}^{t_{n+1}}\wXt^{m-1}dt.\]
	Taking conditional expectations show
	\be \label{e:condmean}
	E^{h,X_0}\big(\what{X}^m_{\tau}|\cF_{t_n}\big) \leq \wXn^m + K E^{h,X_0}\Big( \int_{t_n}^{t_{n+1}}\wXt^{m-1}dt\big| \cF_{t_n}\Big),
	\ee
	yielding (\ref{e:ui}).
\end{enumerate}

To establish (\ref{e:hinvctrl}) we need the following lemma.
\begin{lemma} \label{l:loseexp} Suppose that $h$ satisfies the conditions of Lemma \ref{l:dcomp}, $\sigma$ is bounded, and consider the BEM scheme defined by (\ref{e:BEM}). For any $p\in [0,1)$, any $n$ and $t_n \leq s \leq t <t_{n+1}$ we have
	\[
	E^{h,X_0}\left[h^{-p}(\wXt)|\cF_s\right]\leq h^{-p}(\wXs)\exp(K(t-s)),
	\]
	for some constant $K>0$ that is independent of $n$.
\end{lemma}
\begin{proof}
Let $\mu_t:=\mu(t_n, \wXn; t, \wXt)$. A straightforward application of Ito's formula yields
	\[
	\begin{split}
	dh^{-p}(\wXt)&=dM_t-\frac{\sigma^2(\wXn)}{H_x^{2}(t_n,\wXn;t,\wXt)}ph^{-p}(\wXt)\left(\frac{2\mu_t h'(\wXt)+h''(\wXt)}{2h(\wXt)}+\frac{1-p}{2}\left(\frac{h'}{h}(\wXt)\right)^2\right)dt\\ 
	&\leq dM_t-\frac{\sigma^2(\wXn)}{H_x^{2}(t_n,\wXn;t,\wXt)}ph^{-p}(\wXt)\left(-\frac{\alpha_1}{h(\wXt)}+\frac{(1-p)\alpha_2}{h^2(\wXt)}\right)dt
\end{split}
\]
	where $M$ is a local martingale and $\alpha_1$ and $\alpha _2$ are positive constants depending on the bounds on $h'$ and $h''$ since $\mu_t>c_1$ (resp. $\mu_t<c_2$) whenever $h'(\wXt)>0$ (resp. $h'(\wXt)<0$) by Lemma \ref{l:dcomp} and $h'$ never vanishes at the same time as $h$.  Thus, there exists a constant $K$ that depends only on $h, p$ and $c_1$ and $c_2$ such that 
	\[
dh^{-p}(\wXt)\leq dM_t+K \frac{\sigma^2(\wXn)h^{-p}(\wXt)}{H_x^{2}(t_n,\wXn;t,\wXt)}dt
	\]
since $-\alpha_1 x + (1-p)\alpha_2 x^2$ is bounded from below. 
	
	Next note that $\tau_k:=\inf\{t\geq t_n: \wXt<1/k\}$ is a localising seqeunce for $M$. Thus, using the optional stopping theorem and Fatou's lemma and monotone convergence we arrive at
	\[
	E^{h,X_0}\left[h^{-p}(\wXt)|\cF_s\right]\leq h^{-p}(\wXs)+ K E^{h,X_0}\left[\int_s^th^{-p}(\what{X}_u)du\Big|\cF_s\right], \quad t_n\leq s \leq t\leq t_{n+1},
	\]
	for some constant $K$ in view of the boundedness of $\sigma$.
	We deduce the claim by Gronwall's lemma.
\end{proof}

Now we return to the proof of the estimate (\ref{e:hinvctrl}). 

Observe that the hypothesis on $h''$ implies $1-\exp((s-t_n)\sigma^2(\wXn)\frac{h''}{2h}(\wXn))\leq K\frac{T}{N}\frac{1}{h^p}(\wXn)$ for some $K>0$. Without loss of generality let's also suppose that $h\leq 1$. Thus,
\[
\begin{split}
&E^{h,X_0}\left(\int_{t_n}^{t_{n+1}}\left(1-\exp\big((s-t_n)\sigma^2(\wXn)\frac{h''}{2h}(\wXn)\big)\right)\frac{\sigma^2(\wXn)h^{-p}(\wXs)}{H_x^2(t_n,\wXn;s,\wXs)}ds\right)\\
&\leq K\frac{T}{N}E^{h,X_0}\left(\int_{t_n}^{t_{n+1}}h^{-p}(\wXn)h^{-p}(\wXs)ds\right)\\
&\leq K\frac{T}{N}E^{h,X_0}\left(\int_{t_n}^{t_{n+1}}h^{-2p}(\wXn)ds\right) \leq K\frac{T^2}{N^2}E^{h,X_0}(h^{-1}(\wXn)),
\end{split}
\]
where the second line follows from  Lemma \ref{l:loseexp} and that $H_x\geq 1$.

Next suppose $m\geq 1$. Note that the calculations similar to the ones leading to (\ref{e:condmean}) imply that
\[
E^{h,X_0}\big(\what{X}^m_{t}|\cF_{t_n}\big) \leq \wXn^m + K E^{h,X_0}\Big( \int_{t_n}^{t_{n+1}}\wXs^{m-1}ds\big| \cF_{t_n}\Big),
\]
Thus, the elementary inequality $x^{m-1} \leq 1+ x^m$ and Gronwall's lemma yield
\[
E^{h,X_0}(\wXt^m|\cF_{t_n})\leq K(\wXn + \frac{T}{N}).
\]
Therefore,
\[
\begin{split}
&E^{h,X_0}\left(\int_{t_n}^{t_{n+1}}\left(1-\exp\big((s-t_n)\sigma^2(\wXn)\frac{h''}{2h}(\wXn)\big)\right)\frac{\sigma^2(\wXn)\wXs^m}{H_x^2(t_n,\wXn;s,\wXs)}ds\right)\\
& \leq K\frac{T}{N} E^{h,X_0}\left[ (\wXn +  \frac{T}{N})\left(1-\exp\big(-\frac{T}{N}\frac{a}{h}(\wXn)\big)\right) \right]\\
&\leq K\frac{T^2}{N^2} E^{h,X_0} \left( (\wXn + \frac{T}{N})\frac{1}{h(\wXn)} \right) \leq  K\frac{T^2}{N^2}E^{h,X_0} (\wXn +h^{-1}(\wXn)),
\end{split}
\]
where the last line follows from (\ref{e:loseprod}).

Combining above estimates, we arrive at the claimed result via (\ref{e:integrable1}).
\end{document}